\documentclass[preprint]{elsarticle}

\usepackage[utf8]{inputenc}
\usepackage[english]{babel}

\usepackage[ruled,vlined,algo2e]{algorithm2e}
\usepackage{graphicx}
\usepackage{fancyhdr}

\usepackage{amsfonts}
\usepackage{amssymb}
\usepackage{amsthm}
\usepackage{bm}

\usepackage[normalem]{ulem}
\usepackage{amsmath}
\usepackage{amsfonts}
\usepackage{amssymb}
\usepackage{mathtools}
\usepackage{enumitem}
\usepackage[percent]{overpic}
\usepackage{tikz}
\usepackage{mathrsfs}
\usepackage{xargs}
\usepackage{xcolor}
\usepackage[colorinlistoftodos,prependcaption,textsize=tiny]{todonotes}
\newcommandx{\at}[2][1=]{\todo[linecolor=red,backgroundcolor=red!25,bordercolor=red,#1]{#2}}
\usepackage[font=small,skip=5pt]{caption}

\usepackage{lmodern}
\usepackage[T1]{fontenc}

\usepackage{mathtools,booktabs}

\usepackage{varwidth}
\usepackage{hyperref}
\usepackage{cleveref}
\usepackage{comment}

\DeclareMathOperator{\R}{\mathbb{R}}
\DeclareMathOperator{\F}{\mathbb{F}}
\DeclareMathOperator{\C}{\mathbb{C}}

\DeclareMathOperator{\rank}{\text{rank}}

\usepackage[normalem]{ulem}

\newcommand{\ignore}[1]{}

\newtheorem{theorem}{Theorem}[]
\newtheorem{lemma}[theorem]{Lemma}

\newtheorem{definition}[theorem]{Definition}
\newtheorem{proposition}[theorem]{Proposition}
\newtheorem{corollary}[theorem]{Corollary}

\graphicspath{ {./Graphics/} }

\journal{Linear Algebra and its Applications}

\begin{document}

\begin{frontmatter}


\title{Fast Randomized Numerical Rank Estimation for Numerically Low-Rank Matrices}

\author{Maike Meier\corref{cor1}}
\ead{meier@maths.ox.ac.uk}

\author{Yuji Nakatsukasa}
\ead{nakatsukasa@maths.ox.ac.uk}

\address{Mathematical Institute, University of Oxford,  Oxford, OX2 6GG, UK}

\cortext[cor1]{Corresponding author}



\begin{abstract}
Matrices with low-rank structure are ubiquitous in scientific computing. Choosing an appropriate rank is a key step in many computational algorithms that exploit low-rank structure. However, estimating the rank has been done largely in an ad-hoc fashion in large-scale settings.  
In this work we develop a randomized algorithm for estimating the numerical rank of a (numerically low-rank) matrix. The algorithm is based on sketching the matrix with random matrices from both left and right; the key fact is that with high probability, the sketches preserve the orders of magnitude of the leading singular values.  We prove a result on the accuracy of the sketched singular values and show that gaps in the spectrum are detected.
For an $m\times n$ $(m\geq n)$ matrix of numerical rank $r$, the algorithm runs with complexity $O(mn\log n+r^3)$, or less for structured matrices. The steps in the algorithm are required as a part of many low-rank algorithms, so the additional work required to estimate the rank can be even smaller in practice.  Numerical experiments illustrate the speed and robustness of our rank estimator.

\end{abstract}




\end{frontmatter}


\section{Introduction}\label{sec:intro}
Low-rank matrices are ubiquitous in scientific computing and data science. 
Sometimes a matrix of interest can be shown to be of numerically low rank~\cite{beckermann2017singular,udell2019big}, for example by showing that the singular values decay rapidly. 
More often, matrices that arise in applications may have a hidden low-rank structure, such as low-rank off-diagonal blocks~\cite{martinsson2019fast,MartinssonTroppacta}. 
As is well known, low-rank approximation is also the basis for principal component analysis. 

Numerous studies and computational algorithms exploit the (approximately) low-rank structure of these types of matrices to devise efficient algorithms. Such algorithms usually require finding a low-rank approximation of a given matrix.  For large-scale matrices, traditional deterministic algorithms for low-rank approximation based on computing a truncated singular value decomposition (SVD) might be infeasible. This can be caused by the sheer scale of the matrix and the complexity of classical algorithms, as well as the fact that access to the matrix may be restricted because it is stored in RAM or cannot be stored at all (streaming model)~\cite{MartinssonTroppacta}.  Randomized algorithms have become a powerful and reliable tool for efficiently computing a near-best low-rank approximation for such matrices.  
Landmark references on randomized low-rank factorizations are~\cite{woodruff2014sketching} and ~\cite{halko2011finding}, which extensively analyze the randomized SVD.

A key step in low-rank algorithms, including randomized SVD, is the determination of the numerical rank (based on the spectral norm). 
For instance, a variety of low-rank factorization techniques require the target rank of the factorization as input information. 
Once a low-rank approximation $\hat A_r$ of the specified target rank has been computed, a posteriori estimation of the error $\|A-\hat A_r\|_2$ via a small number of matrix-vector multiplications~\cite{MartinssonTroppacta} is a reliable means of checking if the input rank $r$ was sufficient. If $r$ was too low, one would need to sample the matrix with more vectors. This clearly requires more computational work, and can be a major difficulty in the streaming model, wherein revisiting the matrix is impossible~\cite{tropp2019streaming}. Conversely, if the input rank $r$ was too high, the computational cost of computing $\hat A_r$ would be higher than it could be. 
Having a fast and reliable rank estimator is therefore highly desirable.

In this work we propose and analyze a fast randomized algorithm for estimating the numerical rank of an approximately low-rank matrix $A\in\mathbb{R}^{m\times n}$ or $A\in\mathbb{C}^{m\times n}$.  The algorithm is based on sketching both the column as well as the row space of $A$, forming $Y AX$, where $Y$ and $X$ are random (subspace embedding) matrices. 
The key idea is that with high probability, the singular values of $AX$ are good estimates of the (leading) singular values $\sigma_i(A)$ of $A$. Therefore the decay of $\sigma_i(A)$ can be reliably estimated by the decay of $\sigma_i(AX)$. 
To estimate $\sigma_i(AX)$ we once again sketch $AX$ to obtain the much smaller matrix $Y AX$; it is only $\sigma_i(YAX)$ (or their estimates) that we actually need to compute.  It is noteworthy that in the algorithm it is only necessary to view to matrix $A$ once. Furthermore, since we obtain estimates $\sigma_i(YAX)$ for the leading singular values $\sigma_i(A)$, we will be able to detect a gap in the spectrum of $A$ and hence have the possibility of setting a  tolerance for the numerical rank in a data-driven manner.  

To our knowledge, the fact that a sketch of the form $YAX$ preserves `some coarse spectral structure' was first noted in~\cite{Andoni2013}. We use techniques similar to their proofs for our theoretical results.  Related results for the bulk of the eigenvalues of $(AX)^*(AX)$ were first obtained in~\cite{Gittens2011}.  The main theoretical contributions of this work, deterministic and probabilistic bounds on the accuracy of $\sigma_i(AX)$ as estimates for $\sigma_i(A)$, build on these results.  In particular, we provide a tighter lower bound on $\sigma_i(AX)/\sigma_i(A)$ relevant for small singular values of $AX$.

The complexity of our algorithm is $O(mn\log n+r^3)$ for dense $m\times n$ matrices, and can be lower if $A$ has structure that can be exploited for computing the sketches $AX$ and $Y AX$. 
This is clearly a subcubic complexity (assuming $r\ll m,n$), and it runs significantly faster than computing the full SVD. \ Moreover, in many cases, computing $AX$ (which is usually the dominant part of our rank estimation algorithm) is a required part of the main algorithm (e.g. randomized SVD); 
and in some algorithms~\cite{clarkson2017low,nakatsukasa2020fast} this is true even of $Y AX$. 
In such cases, the additional work needed for estimating the rank is therefore significantly smaller, such as $O(r^3)$ or sometimes even $O(r)$.  Our algorithm is competitive when the target rank $r$ is much smaller than the dimensions of the matrix.

In particular, for the widely used randomized SVD we argue that it is computationally efficient to compute $YAX$ and its singular values, after having computed a sketch $AX$. The reason is that $\sigma_i(YAX)$ will provide information on how appropriate the size of the sketch $AX$ is. Subsequently, more columns can be added or columns can be removed before computing a QR decomposition $QR = AX$ and computing the product $B = Q^*A$.  This will save computational cost in either case: an unnecessary large sketch, or a sketch of insufficiently large dimension. We also introduce an algorithm that combines the rank estimation algorithm with randomized SVD, using the information provided by $\sigma_i(YAX)$.

{\it Notation.} Unless specified otherwise,  $A$ is an $m\times n$ matrix where $m\geq n$.  The $i$th largest singular value of $A$ is denoted by $\sigma_i(A)$, and we furthermore use $\sigma_{\max}(A)$ and $\sigma_{\min}(A)$ for the largest and smallest singular values respectively. We use $\|\cdot\|$ to denote the spectral norm $\|A\|=\sigma_1(A)$, and $\|A\|_F$ is the Frobenius norm.
$\F$ denotes the field, in our case $\F=\R$ or $\F=\C$. The  numerical rank estimate will be denoted by $\hat{r}$. 

Throughout the paper we use $X$ and $Y$ for \emph{random oblivious subspace embedding matrices}, such that with high probability, 
for any fixed $Q$ with orthonormal columns we have $\sigma_{\max}(Q^*X),\sigma_{\max}(YQ) \leq 1 +\tilde{\epsilon}$ and $\sigma_{\min}(Q^*X),\sigma_{\min}(YQ) \geq 1 - \tilde{\epsilon}$ for some $\tilde{\epsilon}\in(0,1)$.  We use $r$, $r_1$ and $r_2$ 
 to refer to the size of embedding matrices, which must be at least the number of columns in $Q$. 
The matrix $X$ is required to have more columns than $Q$, and the same goes for the rows of $Y$. For brevity we simply call such $X$ and $Y$ embedding matrices.

The analysis will be specified to Gaussian or subsampled randomized trigonometric transform (SRTT)~\cite{MartinssonTroppacta} embedding matrices at times. We will use $G\in\R^{n\times r}$,  where $n>r$, to denote a matrix where each entry is iid $\mathcal{N}(0,1)$ distributed and call this matrix a real (standard) Gaussian matrix. In case of $\mathbb{F} = \mathbb{C}$, $G\in\mathbb{C}^{n\times r}$ denotes a complex Gaussian matrix with iid $\mathcal{CN}(0,1)$ entries.  That is, $\operatorname{Re}(G_{ij}),\operatorname{Im}(G_{ij}) \sim \mathcal{N}(0,1/2)$ and independent. A Gaussian matrix can be scaled to an embedding matrix by defining $X = G/\sqrt{r}$ (in both the real and complex case); we call the scaled Gaussian matrices \emph{Gaussian embedding matrices}.  SRTT matrices are defined in Section \ref{sec:choicesketch} and as the scaling is included in the definition,  they are naturally embedding matrices. We will use $\Theta$ to denote an SRTT matrix.

\subsection{Numerical rank and goal of a rank estimator}\label{sec:whatisgoal}
So far we have been using the term ``the numerical rank'' informally. Let us now define the notion. 
This is a standard definition, see for example~\cite{beckermann2017singular}. 
\begin{definition}
Let $A\in\mathbb{F}^{m\times n}$. The $\epsilon$-rank of $A$, denoted by $\rank_\epsilon(A)$, is the integer\footnote{If no such $i$ exists, we take $\rank_\epsilon(A)=\min(m,n)$; however, in this paper we are never interested in this full-rank case. 
We also note that MATLAB's {\tt rank} function takes the \emph{absolute} $\epsilon$-rank, i.e., $\sigma_i(A)>\epsilon\geq \sigma_{i+1}(A)$  for a user-specified tolerance $\epsilon$}.
  $i$ such that $\sigma_i(A)>\epsilon\|A\|_2 \geq \sigma_{i+1}(A)$.
\end{definition}
We adopt the \textit{relative} $\epsilon$-rank definition as it is a natural choice in the context of low-rank matrix approximation.  One could easily adapt the algorithm presented to estimate the \textit{absolute} $\epsilon$-rank.
Let us discuss what the goal of a rank estimator should be. One natural answer of course is that the estimator should output $\rank_\epsilon(A)$ given $A$ and the (user-defined) relative threshold $\epsilon$. However, we argue that the situation is more benign: consider, for example, a matrix with $\|A\| = 1$ and $k$ singular values $>10\epsilon$, five singular values 
in $(\epsilon,1.1\epsilon)$, five more in 
$(0.9\epsilon,\epsilon)$, and the remaining $n-k-10$ are $<0.01\epsilon$. What should the estimator output? Is it crucial that the ``correct'' value $\rank_\epsilon(A)=k+5$ is identified? 

In this paper we take the view that the goal of the rank estimator is to find a good $\epsilon$-rank, not necessary \emph{the} $\epsilon$-rank. In the example above, any number between $k$ and $k+10$ would be an acceptable output. 
In most applications that we are aware of (including low-rank approximation, regularized linear systems, etc), there is little to no harm in choosing a rank $\hat r\neq \rank_\epsilon$:
a slight overestimate $\hat r> \rank_\epsilon$ usually results in slightly more computational work, whereas 
a slight underestimate $\hat r< \rank_\epsilon$ is fine if $\sigma_{\hat r+1}(A)=O(\epsilon \|A\|)$, that is, a rank-$\hat r$ matrix can approximate $A$ to relative $O(\epsilon)$-accuracy. 

On the other hand, it is clearly \emph{not} acceptable if the rank is unduly underestimated in that $\sigma_{\hat r+1}(A)\gg \epsilon\|A\|$. 
It is also unacceptable if $\sigma_{\hat r}(A)\ll \epsilon\|A\|$. 
The goal in this paper is to devise an algorithm that efficiently and reliably finds an $\hat r$ such that 
\begin{itemize}
\item $\sigma_{\hat r+1}(A)=O(\epsilon\|A\|)$ (say, $\sigma_{\hat r+1}(A)<10\epsilon \|A\|$): $\hat r$ is not a severe underestimate, and 
\item $\sigma_{\hat r}(A)=\Omega(\epsilon\|A\|)$ (say, $\sigma_{\hat r}(A)>0.1\epsilon \|A\|$): $\hat r$ is not a severe overestimate.
\end{itemize}

Any such $\hat r$ is sensible in that there exists a rank-$\hat r$ approximation of $A$ with $O(\epsilon)$ relative accuracy, and $\hat r$ is not much larger than the best possible for the accuracy required.
In addition, the nature of randomized algorithms means that the user must be willing to accept results that do not hold precisely and that small inaccuracies in the results are acceptable. As a result, we assume one is not looking for the exact numerical rank, but an order-of-magnitude estimate.
Our rank estimate will satisfy these conditions with high probability. 
In particular, in situations where the numerical rank is clearly defined, i.e., 
if a clear gap is present in the spectrum $\sigma_r(A)\gg \epsilon\|A\| \gg \sigma_{r+1}(A)$, the algorithm will reliably find the exact rank $\hat r = \rank_\epsilon(A)$. 

Additionally, in situations where it is unknown whether the matrix $A$ is low-rank approximable, our rank estimator can tell us (roughly) how well $A$ can be approximated with a low-rank matrix.  Similarly,  if it is unknown what a suitable tolerance would be and/or the user would like to detect a gap in the spectrum,  the algorithm can be used to plot an estimated spectrum and detect gaps.

This paper is organized as follows. Section~\ref{sec:related} discusses related studies in the literature. In Section~\ref{sec:theory} we show that $\sigma_i(AX)$ gives useful information about $\sigma_i(A)$ for leading values of $i$. Then in Section~\ref{sec:approxorth} we show that $\sigma_i(AX)$ can be estimated via $\sigma_i(YAX)$. Section~\ref{sec:mainalg} summarizes the overall rank estimation algorithm. Numerical experiments are presented in Section~\ref{sec:exp}. 

\section{Related work}\label{sec:related}
\subsection{Existing methods for numerical rank estimation}
At the core of numerical rank estimation lies singular value estimation; specifically,  the rank can be found by counting the number of singular values greater than the tolerance $\epsilon$. The most direct way to do this is to explicitly compute the singular values of a matrix.  However, as is well known, this requires $O(mn^2)$ operations for an $m\times n$ matrix~\cite{golubbook4th}, and for large matrices this is computationally inadmissible.  Another point of view, specifically for Hermitian matrices, is that a numerical rank estimate can follow from estimating the density of states (DOS).  The DOS can be interpreted as a probability density function describing the position of the eigenvalues, and can be estimated with algorithms analyzed in~\cite{specdensity2016}. 

Regarding fast algorithms that can run with subcubic complexity,  the literature on rank estimation appears rather scant.  Exceptions include the work of Ubaru and Saad~\cite{ubaru2016fast}, Ubaru, Saad and Seghouane~\cite{Ubaru2017},  Zhang, Wainwright and Jordan~\cite{zhang2015distributed} and Di Napoli, Polizzi and Saad \cite{di2016efficient}.  These works are all concerned with counting the number of eigenvalues in a certain interval for a Hermitian matrix. Alternatively, \cite{cheung2013fast} discusses an algorithm for computing the exact rank of a matrix.

The first two references 
\cite{ubaru2016fast,Ubaru2017}
employ the idea of density of states (DOS). In~\cite{ubaru2016fast}, the authors use the DOS to locate a gap in the spectrum, derive an appropriate tolerance $\epsilon$ and subsequently use polynomial approximation and stochastic trace estimation to count the number of eigenvalues greater than the tolerance. Computationally, the algorithm only requires matrix-vector products with $A$ but consequently also requires many views of $A$.  In the second paper~\cite{Ubaru2017}, the authors also first estimate the DOS to locate a gap and then estimate the integral of the DOS.  We compare against these algorithms in the numerical experiments, although they need to be adjusted slightly to account for matrices that are not symmetric positive semi-definite.  An advantage of both methods is that the entire spectrum is estimated,  however they require many more views of the matrix than most randomized methods.  The results are discussed in Section \ref{sec:expcompRE}. The works~\cite{zhang2015distributed} and~\cite{di2016efficient} are similar in spirit. The first focuses on minimizing communication cost involved with rank estimation in a distributed setting, the second approximates an eigenvalue count using polynomial and rational approximation filtering.

As for more theoretical contributions, much of this work builds upon Andoni and Nguy\~{\^{e}}n~\cite{Andoni2013}, which is to the best of the authors' knowledge the first work that shows rank can be estimated via sketches.  The work is focused on estimating eigenvalues of Hermitian matrices via the sketch $X^TAX$, but also considers the singular values of non-symmetric matrices using $YAX$. 
While Andoni and Nugyen work with the Gram matrix of the sketch, obtaining results in terms of the difference of the squared singular values (or via the Jordan-Wieldant matrix, which is larger), we work with $YAX$ directly. 
Moreover, their work is of a theoretical nature, and do not present a concrete algorithm for estimating the rank. They also state results that only come with relatively low probability guarantees, although it is possible to take a larger sample to improve the probability. 

Here we introduce precise algorithms (also for low-rank approximation), show numerical experiments, and phrase our results in terms of popular embedding matrices allowing for results that hold with exponentially small failure probability.

\subsection{Related work in statistics}
Rank estimation, or singular value estimation, is a well-known problem in the context of covariance matrix estimation or PCA in statistics, or the detection of signals in signal processing.  Most rank estimation algorithms in this context require computing an SVD or eigendecomposition; exceptions include~\cite{ubaru2019find}. Since we consider a context where computing these decompositions exactly is not computationally feasible, we do not take this into account. 

The statistics literature is connected to this work in yet another way: relating the singular values of the sketch $AX$ to the singular values of $A$ is equivalent to the problem of covariance matrix estimation.  This can be seen as follows: let $A = U\Sigma V^*$ be the SVD of $A$ and note the singular values of $AX$ are the same as the singular values of $\Sigma X_1$, where $X_1 = V^*X$ is a Gaussian embedding matrix. The columns of $\Sigma X_1$ can be viewed as $r$ scaled observations from an $n$-dimensional $\mathcal{N}(0, \Sigma^2)$ distribution. As the number of observations tends to infinity, the singular values of the 'data matrix' $\Sigma X_1$ tend to the singular values of $\Sigma$.  Estimation of (approximately) low-rank covariance matrix is studied in, among others, \cite{vershynin2018high, wainwright2019high, koltchinskii2017concentration} and has a clear relation to PCA.

One specific related example in the statistics literature that is well studied is the spiked covariance matrix model, introduced by Johnstone~\cite{johnstone2001distribution} and analyzed by, among others, 
Bai and Silverstein~\cite{bai2010spectral},
Nadler \cite{Nadler2008finite} and Rao et al. \cite{Rao2008StatisticalMatrices}.  We find that the theory in these works cannot be applied in our context, since the tail of the singular values in this model is very heavy.  In the general context of covariance matrix estimation, various authors have suggested different manipulations to the sample covariance matrix or its eigenvalues to improve the estimator, known as shrinkage. See, for instance, the work by Ledoit and Wolf on linear shrinkage \cite{Ledoit2004AMatrices}, nonlinear computational shrinkage \cite{Ledoit2012NonlinearMatrices} and nonlinear analytical shrinkage \cite{Ledoit2020AnalyticalMatrices}, or by El Karoui based on random matrix theory \cite{ElKaroui2008SpectrumTheory}. 
Our numerical experiments suggest these methods are either inefficient for large-scale matrices or unsuitable for numerically low-rank matrices, and we do not discuss them further.

\subsection{Related algorithms in (randomized) numerical linear algebra}
A number of recent papers focus on randomized algorithms to perform numerical linear algebra tasks on large matrices.  Specifically, randomized algorithms for either full or low-rank matrix factorization are of interest in the discussion of numerical rank estimation. This relevance is two-fold, as a numerical rank estimate can be a natural consequence of a factorization. In other cases, an a priori rank estimate can aid in performing a (low-rank) factorization. 

The former is the case for rank-revealing full factorizations, such as rank-revealing QR decompositions \cite{Duersch2020randomized} or UTV factorizations \cite{Martinsson2019RandUTV}.  Full factorizations are (relatively speaking) most useful when the $\epsilon$-rank is not much smaller than the matrix dimensions, and still require cubic $O(mn^2)$ complexity.  As we mainly consider the case where $\rank_\epsilon(A)\ll \min(m,n)$, we focus on methods that have subcubic complexity.

A number of randomized algorithms for a low-rank factorization have been proposed  and they can often be modified to yield a (rank $r$) QB factorization: $A \approx QB$, where $Q\in\F^{m\times r}$ is a matrix with orthonormal columns and $B\in\F^{r\times n}$ is a short and fat matrix. 
Singular value estimates can then be obtained by computing the exact singular values of $B$, which in turn leads to a numerical rank estimate. Within randomized algorithms for QB factorizations, we can distinguish between (adaptive) algorithms that focus on solving the fixed-precision problem and non-adaptive algorithms that focus on the fixed-rank problem.

The fixed-precision problem can be summarised as: find $QB$ such that $\|A-QB\|\leq \epsilon\|A\|$. Adaptive algorithms for this problem include the adaptive randomized range finder~\cite{halko2011finding}, the incremental rangefinder \cite{MartinssonTroppacta},  RandQB\_b~\cite{martinsson2016randomized} and RandQB\_EI~\cite{Yu2018efficient}. The fixed-precision problem is discussed in more detail in Section \ref{sec:fixedprecision}.  

Conversely, the fixed-rank problem is: find $QB$ of exact rank $r$ such that $\|A-QB\|$ is as small as possible.  One of the most well-known algorithms in randomized fixed-rank factorization is randomized rangefinder~\cite{halko2011finding}, which is a core part of the randomized SVD algorithm proposed in~\cite{halko2011finding}.  

We discuss these algorithms further, comparing with our algorithm, in Sections~\ref{sec:mainalg}--\ref{sec:exp}, to  demonstrate that our algorithm is much more efficient for rank estimation, and can often be used as a convenient preprocessing step for a low-rank approximation algorithm to determine the input rank. 

Conversely, a numerical rank estimate can aid the factorization process. This is discussed in detail in Section \ref{sec:fixedprecision} and forms the basis for a non-adaptive low-rank factorization algorithm we propose for the fixed-precision problem.

\section{Sketching roughly preserves singular values: \boldmath$\sigma_i(AX)/\sigma_i(A)=O(1)$ for leading $i$}\label{sec:theory}
In this section we show that if a matrix $A$ has low numerical rank, then sketching preserves its leading singular values sufficiently well to provide an estimate for its numerical rank. Again drawing the connection to sample covariance matrix estimation, it is intuitive that it should be difficult, if not impossible, to obtain information about an $n$-dimensional distribution using only $r$ samples, when $r\ll n$. Clearly, there are $n-r$ singular values (or signals) we cannot detect, but it is also not clear why the $r$ singular values we estimate would be any good. One way to think about this is to see that $\mathcal{N}(0, \Sigma^2)$ lies close to a low-dimensional subspace, since $A$ is approximately low-rank. This makes it more intuitive why a small sample size could suffice to retrieve a reasonably accurate approximation of $\Sigma$ (see also~\cite[Sec.  9.2.3]{vershynin2018high}). 

We make the connection between the low-rank structure of $A$ and our ability to estimate its singular values using $AX$ explicit with the deterministic and probabilistic error bounds on the ratio $\sigma_i(AX)/\sigma_i(A)$.  The analysis will focus on Gaussian matrices. As explained before, $G\in\R^{n\times r}$ denotes a Gaussian matrix with iid $\mathcal{N}(0,1)$ ($\F = \R$) or $\mathcal{CN}(0,1)$ ($\F=\C$) entries, and a Gaussian embedding matrix will be of the form $X = G/\sqrt{r}$.  We start from the unscaled sketch $AG$.  Let $A=U\Sigma V^*$ be the SVD of $A$ and decompose this further to $A = U_1\Sigma_1V_1^* + U_2\Sigma_2V_2^*$, where $U_1$ contains the leading $r$ singular vectors of $A$.
Now decompose $AG\in\F^{m\times r}$ as 
\begin{align}\label{eq:AGV1V2}
    AG &= U_1\Sigma_1(V_1^*G) + U_2\Sigma_2(V_2^*G) = U_1\Sigma_1G_1 + U_2\Sigma_2G_2,
\end{align}
where $\Sigma_1$ is $r\times r$ and $G_1\in\mathbb{F}^{r \times r}$ and $G_2\in\mathbb{F}^{(n-r) \times r}$ are independent Gaussian matrices. 
Define $B_1 = U_1\Sigma_1G_1 \in \F^{m\times r}$ and $B_2 = U_2\Sigma_2G_2\in\F^{m\times r}$. We start with the following result, which examines the relation between $\sigma_i(AG)$ and $\sigma_i(A)$. 
\begin{lemma}\label{lemma:deterrAX} 
Let $A\in\F^{m\times n}$ and $G\in\F^{n\times r}$. Decompose $AG$ as in~\eqref{eq:AGV1V2}. Then, for $i = 1,\dots, r$,
\begin{align}
    \sigma_{\min}(\hat{G}_{\{i\}}) &\leq \frac{\sigma_i(AG)}{\sigma_i(A)} \leq \sqrt{\sigma_{\max}(\tilde{G}_{\{r-i +1\}})^2 + \left(\frac{\sigma_{r+1}(A)\sigma_{\max}(G_2)}{\sigma_i(A)}\right)^2}, \label{eq:4.detbound}
\end{align}
where $\hat{G}_{\{i\}}$ is an $i\times r$ matrix consisting of the first $i$ rows of $G_1$, and $\tilde{G}_{\{r-i +1\}}$ is the matrix of the last $r-i+1$ rows of $G_1$. 
Furthermore, if $G$ is a Gaussian matrix, then
for each $i\in\{1,\ldots,r\}$, the random matrices $\hat{G}_{\{i\}}$, $\tilde{G}_{\{r-i+1\}}$, and $G_2$ are independent Gaussian in $\F$.
\end{lemma}
\begin{proof}
The proof consists of establishing the identities
\begin{align}
    1 &\leq \frac{\sigma_i(AG)}{\sigma_i(B_1)} \leq \sqrt{1 + \left(\frac{\sigma_{r+1}(A)\sigma_{\max}(G_2)}{\sigma_i(B_1)}\right)^2}, \label{eq:4.detbound1} \\
    \sigma_{\min}(\hat{G}_{\{i\}}) &\leq \frac{\sigma_i(B_1)}{\sigma_i(A)} \leq \sigma_{\max}(\tilde{G}_{\{r-i +1\}}),  \label{eq:4.detbound2} 
\end{align}
for $i= 1,\dots, r$. 

It is immediate from the definitions of $B_1$ and $B_2$ that $(AG)^*(AG) = B_1^*B_1 + B_2^*B_2$. As a result, we have by Weyl's theorem, for $i = 1,\dots,r$, 
\begin{equation}
	\sigma_i(B_1)^2 + \sigma_{\min}(B_2)^2 \leq \sigma_i(AG)^2 \leq \sigma_i(B_1)^2 + \|B_2\|^2.
\end{equation}
The upper bound follows from observing $\|B_2\|\leq\|\Sigma_2\|\|G_2\| = \sigma_{r+1}(A)\sigma_{\max}(G_2)$. The lower bound is immediate.

We use the interlacing property of singular values to show~\eqref{eq:4.detbound2}. For a square matrix $B\in\F^{k\times k}$, let $B_{[-p]}\in\F^{k\times (k-p)}$ denote a submatrix of $B$ obtained by deleting any $p$ columns. Then~\cite{mathias1990two}
\begin{equation}\sigma_j(B)\leq\sigma(B_{[-p]})\leq \sigma_{j+p}(B), \quad \text{ for } j = 1,2,\dots, k-p.\label{eq:interlacingsvs}
	\end{equation}
First, notice the upper bound~\eqref{eq:4.detbound2} is trivial for $i = 1$. For $i = 2,\dots,r$, partition $G_1^*\Sigma_1$ as 
$$G_1^*\Sigma_1 = \left[\hat{G}^*_{\{i-1\}}\Sigma_{11}\,\,\,\tilde{G}^*_{\{r-i+1\}}\Sigma_{12} \right],$$
where $\Sigma_{11}\in\R^{(i-1)\times(i-1)}$ and $\Sigma_{12}\in\R^{(r-i+1)\times (r-i+1)}$ are diagonal matrices which contain the singular values of $\Sigma_1$. We then have
\begin{align*}
	\sigma_i(B_1) &= \sigma_i(
                        \begin{bmatrix}
\hat{G}^*_{\{i-1\}}\Sigma_{11}&\tilde{G}^*_{\{r-i+1\}}\Sigma_{12}                          
                        \end{bmatrix}) \\
	&\leq \sigma_1(\tilde{G}^*_{\{r-i+1\}}\Sigma_{12}) \\
	& \leq \|\Sigma_{12}\|_2\|\tilde{G}_{\{r-i+1\}}\|_2 = \sigma_i(A)\|\tilde{G}_{\{r-i+1\}}\|_2,
\end{align*}
where we applied~\eqref{eq:interlacingsvs} with $j=1$ and $p=i-1$ for the first inequality. Similarly, for the lower bound and $i = 1,\dots,r-1$, consider the partition
$$G_1^*\Sigma_1 = \left[\hat{G}^*_{\{i\}}\Sigma_{11}\,\,\,\tilde{G}^*_{\{r-i\}}\Sigma_{12} \right],$$
where now $\Sigma_{11}$ is $i\times i$ and $\Sigma_{12}$ is $(r-i)\times(r-i)$. We have
\begin{align*}
	\sigma_i(B_1) &= \sigma_i(\left[\begin{smallmatrix}\hat{G}^*_{\{i\}}\Sigma_{11}&\tilde{G}^*_{\{r-i\}}\Sigma_{12} \end{smallmatrix}\right]) \\
	&\geq \sigma_i(\hat{G}^*_{\{i\}}\Sigma_{11}) \\
	& \geq \sigma_{\{i\}}(\Sigma_{11})\sigma_{\min}(\hat{G}_{\{i\}}) = \sigma_i(A)\sigma_{\min}(\hat{G}_{\{i\}}).
\end{align*}
The lower bound is immediate for $i = r$. 

If $G$ is  Gaussian, its orthogonal invariance implies $\big[
\begin{smallmatrix}
  V_1^*\\V_2^*
\end{smallmatrix}
\big]G$ is also Gaussian. 
Since $\hat{G}_{\{i\}}$, $\tilde{G}_{\{r-i+1\}}$, and $G_2$ are submatrices of this matrix with no overlap for a fixed $i$, it immediately follows that they are Gaussian and independent. 
\end{proof}

It is worth noting that~\eqref{eq:4.detbound} holds without the assumption that $G$ is Gaussian.  

Lemma~\ref{lemma:deterrAX} motivates us to look at the singular values of independent Gaussian matrices $\hat{G}_{\{i\}}$, $\tilde{G}_{\{r-i +1\}}$ and $G_2$. To this end we use two classic results from random matrix theory to derive probabilistic bounds on the expected error and on the failure probability. 
\begin{theorem}[Marčenko and Pastur~\cite{Marcenko1967DistributionMatrices}, Davidson and Szarek~\cite{davidson2001local},  Aubrun and Szarek~\cite{aubrun2017alice}]\label{thm:MP} Let $G\in\R^{m\times n}$ be a standard Gaussian matrix with $m\geq n$. Then 
\begin{equation}\label{eq:MP}
    \sqrt{m}-\sqrt{n}\leq \mathbb{E}\,\sigma_{\min}(G)\leq\mathbb{E}\,\sigma_{\max}(G)\leq \sqrt{m} + \sqrt{n}.
\end{equation}
Furthermore, for every $t\geq 0$ one has
\begin{align}\label{eq:DS}
    \mathbb{P}\{\sigma_{\min}(G) \leq \sqrt{m} - \sqrt{n} - t \} \leq e^{-t^2/2}, \quad \mathbb{P}\{\sigma_{\max}(G) \geq \sqrt{m} + \sqrt{n} + t \} \leq e^{-t^2/2}.
\end{align}
Let $G\in\C^{m\times n}$ be a complex Gaussian matrix with $m\geq n$. Then, for every $t\geq 0$ one has
\begin{equation}\label{eq:AScomplexUB}
\mathbb{P}\{\sigma_{\max}(G) \geq \sqrt{m} + \sqrt{n} +t\}\leq \frac{1}{2}e^{-t^2}.
\end{equation}
Furthermore, for every $t> 4\sqrt{2\log n}/(\sqrt{m/n}-1)$,
\begin{equation}\label{eq:AScomplexLB}
\mathbb{P}\{\sigma_{\min}(G) \leq \sqrt{m} - \sqrt{n} -t\}\leq e^{-t^2/4}.
\end{equation}
\end{theorem}
A useful way to interpret these results in simple terms is that rectangular random (Gaussian) matrices with aspect ratio $m/n>1$ are well-conditioned, with singular values supported in $[\sqrt{m}-\sqrt{n},\sqrt{m}+\sqrt{n}\kern1pt ]$. 
We note that~\eqref{eq:MP}, which follows from the Marčenko and Pastur rule~\cite{Marcenko1967DistributionMatrices}, holds more generally for random matrices with iid entries with mean 0 and variance 1.  The second pair of bounds \eqref{eq:DS} do not generally hold for other random matrices, see \cite{Rudelson2010} for details.  As far as the authors are aware, there does not exist a lower bound on the expectation of the smallest singular value of a complex normal matrix. An upper bound for the largest singular value can be found in~\cite[Prop. 6.33]{aubrun2017alice}.
We use these results to bound the singular values of $\hat{G}_{\{i\}}$, $\tilde{G}_{\{r-i+1\}}$,  and $G_2$, which leads to the following theorem. We present the result only in the real case, but it can easily be extend to the complex case by combining Lemma~\ref{lemma:deterrAX} with~\eqref{eq:AScomplexUB} and~\eqref{eq:AScomplexLB}.
\begin{theorem}\label{thm:mainerrAX}
    Let $X\in\R^{n\times r}$ be a Gaussian embedding matrix, i.e., $X = G/\sqrt{r}$, and let $A\in\R^{m\times n}$, where $m\geq n$. 
Then, for $i=1,\dots, r$
    \begin{equation}\label{eq:ineqGauss}
        1- \sqrt{\frac{i}{r}}  \leq \text{\large{$\mathbb{E}$}}\frac{\sigma_i(AX)}{\sigma_i(A)}\leq 1 + \sqrt{\frac{r-i+1}{r}} + \frac{\sigma_{r+1}(A)}{\sigma_i(A)}\left(1 + \sqrt{\frac{n-r}{r}}\right).
    \end{equation}
    Additionally, for each $i=1,\dots,r$ we have, with failure probability at most $3e^{-t^2/2}$,
    \begin{multline}\label{eq:highconcentrate}
    1 - \sqrt{\frac{i}{r}} - \frac{t}{\sqrt{r}} \leq \frac{\sigma_i(AX)}{\sigma_i(A)}\\ 
    \leq 1+ \sqrt{\frac{r-i+1}{r}}  + \frac{\sigma_{r+1}(A)}{\sigma_i(A)}\left(1+ \sqrt{\frac{n-r}{r}} \,\right) + \frac{t}{\sqrt{r}}\left(1 + \frac{\sigma_{r+1}(A)}{\sigma_i(A)}\right).
    \end{multline}
\end{theorem}
\begin{proof}
Note first from Marčenko-Pastur~\eqref{eq:MP}
that we have 
\begin{align*}
\text{\large{$\mathbb{E}$}}\,\sigma_{\min}(\hat{G}_{\{i\}}) &\geq \sqrt{r} - \sqrt{i}; \\
\text{\large{$\mathbb{E}$}}\,\sigma_{\max}(\tilde{G}_{\{r-i+1\}}) &\leq \sqrt{r} + \sqrt{r-i+1}; \\
\text{\large{$\mathbb{E}$}}\,\sigma_{\max}(G_2) &\leq \sqrt{n-r} + \sqrt{r}.
\end{align*}
We can combine this with the result of Lemma \ref{lemma:deterrAX} to find
\begin{align*}
\text{\large{$\mathbb{E}$}}\frac{\sigma_i(AG)}{\sigma_i(A)} &\leq \text{\large{$\mathbb{E}$}}\left[\sigma_{\max}(\tilde{G}_{\{r-i +1\}})^2 + \left(\frac{\sigma_{r+1}(A)\sigma_{\max}(G_2)}{\sigma_i(A)}\right)^2 \right]^{1/2} \\
&\leq \text{\large{$\mathbb{E}$}}\left[\sigma_{\max}(\tilde{G}_{\{r-i +1\}}) + \frac{\sigma_{r+1}(A)\sigma_{\max}(G_2)}{\sigma_i(A)}\right] \\
& = \text{\large{$\mathbb{E}$}}\,\sigma_{\max}(\tilde{G}_{\{r-i +1\}}) +  \frac{\sigma_{r+1}(A)}{\sigma_i(A)}\,\text{\large{$\mathbb{E}$}}\,\sigma_{\max}(G_2) \\
& \leq \sqrt{r} + \sqrt{r - i + 1} + \frac{\sigma_{r+1}(A)}{\sigma_i(A)}\left(\sqrt{n-r} + \sqrt{r}\right),
\end{align*}
and
\begin{align*}
\text{\large{$\mathbb{E}$}}\,\frac{\sigma_i(AG)}{\sigma_i(A)} \geq \text{\large{$\mathbb{E}$}}\,\sigma_{\min}(\hat{G}_{\{i\}}) \geq \sqrt{r} - \sqrt{i}.
\end{align*}
The first result of Theorem \ref{thm:mainerrAX} then follows by applying the scaling $1/\sqrt{r}$.  Similarly,  by \eqref{eq:DS} we have with failure probability at most $3e^{-t^2/2}$ that the following statements hold simultaneously:
\begin{align*}
\frac{\sigma_{\min}(\hat{G}_{\{i\}})}{\sqrt{r}} &\geq 1 - \sqrt{\frac{i}{r}} - \frac{t}{\sqrt{r}}, \\
\frac{\sigma_{\max}(\tilde{G}_{\{r-i+1\}})}{\sqrt{r}} &\leq 1+ \sqrt{\frac{r-i+1}{r}} + \frac{t}{\sqrt{r}}, \\
\frac{\sigma_{r+1}}{\sigma_i}\frac{\sigma_{\max}(G_2)}{\sqrt{r}} &\leq \frac{\sigma_{r+1}(A)}{\sigma_i(A)}\left(1 + \sqrt{\frac{n-r}{r}}\right) + \frac{\sigma_{r+1}(A)}{\sigma_i(A)}\frac{t}{\sqrt{r}}.
\end{align*}
Since the deterministic results in Lemma \ref{lemma:deterrAX} imply
\begin{align*}
\frac{\sigma_{\min}(\hat{G}_{\{i\}})}{\sqrt{r}} \leq \frac{\sigma_i(AG)}{\sigma_i(A)}&\leq \frac{\sigma_{\max}(\tilde{G}_{\{r-i+1\}})}{\sqrt{r}} + \frac{\sigma_{r+1}(A)}{\sigma_i(A)}\frac{\sigma_{\max}(G_2)}{\sqrt{r}},
\end{align*}
the result follows.
\end{proof}
\subsection{Bounds for general subspace embeddings}
The results above focus on Gaussian embedding matrices, and show that the ratios of the singular values $\sigma_i(AX)/\sigma_i(A)$ are reasonably close to $1$, as specified by~\eqref{eq:ineqGauss}. We now show that much of this carries over to a general subspace embedding $X$.  Note that an $r$-dimensional subspace embedding will not generally embed an $r$-dimensional space, but rather a smaller, say $\tilde{r}$-dimensional, space. 
\begin{theorem}\label{thm:embed}
Let $\tilde V_1\in\mathbb{F}^{n\times \tilde r}$ be the leading $\tilde r (\leq r)$ right singular vectors of $A\in\mathbb{F}^{m\times n}$, and 
suppose $X\in\F^{n\times r}$ is a subspace embedding for the subspace $\tilde V_1$ satisfying $\sigma_i(\tilde V_1^TX)\in[1-\tilde{\epsilon},1+\tilde{\epsilon}]$ for some $\tilde{\epsilon}<1$. 
Then, for $i=1,\dots, \tilde r$,
    \begin{equation}\label{eq:ineqembed}
1-\tilde{\epsilon} \leq \frac{\sigma_i(AX)}{\sigma_i(A)}\leq 
\sqrt{(1+\tilde{\epsilon})^2 + \left(\frac{\sigma_{\tilde r+1}(A)\|X\|_2}{\sigma_i(A)}\right)^2}.
    \end{equation}
\end{theorem}
\begin{proof}
We start with a variant of~\eqref{eq:4.detbound}, 
which follows from the same arguments:
\[
    \sigma_{\min}(\hat{X}_{\{i\}}) \leq \frac{\sigma_i(AX)}{\sigma_i(A)} \leq \sqrt{\sigma_{\max}(\tilde{X}_{\{\tilde r-i +1\}})^2 + \left(\frac{\sigma_{\tilde r+1}(A)\sigma_{\max}(X_2)}{\sigma_i(A)}\right)^2}, \quad i=1,\ldots,\tilde r.
\]
Here $\hat{X}_{\{i\}}, \tilde{X}_{\{\tilde r-i +1\}}$  are submatrices of $\tilde V_1^TX$ (as in Lemma~\ref{lemma:deterrAX}), and $X_2 = \tilde V_2^TX$, where $\tilde V_2\in\mathbb{F}^{n\times (n-\tilde r)}$ is the trailing right singular vectors of $A$. 
By the assumption $\sigma_i(\tilde V_1^TX)\in[1-\tilde{\epsilon},1+\tilde{\epsilon}]$, we have 
\[\sigma_{\min}(\hat{X}_{\{i\}})\geq \sigma_{\min}(\tilde V_1^TX)\geq 1-\tilde{\epsilon},\]
which gives the lower bound  in~\eqref{eq:ineqembed}, and also 
\[\sigma_{\max}(\tilde{X}_{\{\tilde r-i +1\}})\leq \sigma_{\max}(\tilde V_1^TX)\leq 1+\tilde{\epsilon}.\]
Finally, $\sigma_{\max}(X_2)=\|\tilde V_2^TX\|_2\leq \|X\|_2$ to obtain the upper bound. 
\end{proof}
We note that the $\|X\|_2$ term in~\eqref{eq:ineqembed} is bounded by $\sqrt{n/r}$, often deterministically, in most commonly-used embeddings (see, e.g.~\eqref{eq:Theta}). 
This is roughly equal to $\sqrt{\frac{n-r}{r}}$ appearing in~\eqref{thm:mainerrAX}. 
Thus Theorems~\ref{thm:mainerrAX} and~\ref{thm:embed} offer roughly the same level of guarantee in terms of the quality of $\sigma_i(AX)$ as an estimate for $\sigma_i(A)$. 
\subsection{The effect of tail singular values and oversampling}
The proof of Lemma \ref{lemma:deterrAX} allows us to see how the tail of the singular values, $\Sigma_2$, affects the lower bound on the ratio $\sigma_i(AG)/\sigma_i(A)$.  This allows us to make a case for oversampling based on singular value estimation, whereas usually oversampling is inspired through a discussion on singular vectors. Again using the Courant-Fischer theorem, we have
\begin{align*}
\sigma_i(AG)&=  \max_{\substack{\mathcal{S}\subset \F^r\\ \dim \mathcal{S} = i}\,\,} \min_{\substack{x\in\mathcal{S} \\ \|x\| = 1}} \left\|
      \begin{bmatrix}
     \Sigma_1G_1\\
     \Sigma_2G_2
     \end{bmatrix}x \right\|.
\end{align*}
Now let $\hat{\mathcal{S}}\subset\F^r$ and $\hat{x}\in\hat{\mathcal{S}}$ be such that 
\begin{align*}
\max_{\substack{\mathcal{S}\subset \F^r\\ \dim \mathcal{S} = i}\,\,} \min_{\substack{x\in\mathcal{S} \\ \|x\| = 1}} \left\|
     \Sigma_1G_1 x \right\| = \min_{\substack{x\in\hat{\mathcal{S}} \\ \|x\| = 1}} \left\|
     \Sigma_1G_1 x \right\| = \left\|
     \Sigma_1G_1 \hat{x} \right\| =\left\|
     B_1 \hat{x} \right\|= \sigma_i(B_1).
\end{align*}
Then $\sigma_i(AG)\geq \| \begin{bmatrix}\Sigma_1G_1 & \Sigma_2G_2 \end{bmatrix}^T \hat{x}\| = \sqrt{\|B_1\hat{x}\|^2 + \|\Sigma_2G_2\hat{x}\|^2}$. We can use this to show
\begin{equation}\label{eq:effecttail}
     \frac{\sigma_i(AX)}{\sigma_i(A)} = \frac{\sigma_i(AG)}{\sigma_i(B_1)}\frac{\sigma_i(B_1)}{\sigma_i(A)\sqrt{r}} \geq \frac{\sigma_{\min}(\hat{G}_{\{i\}}) }{\sqrt{r}}\sqrt{1 + \frac{\|\Sigma_2G_2\hat{x}\|^2}{\sigma_i(B_1)^2}}.   
\end{equation}
As $G_2\hat{x}$ is a Gaussian vector, we know $\mathbb{E}\|\Sigma_2G_2\hat{x}\| = \|\Sigma_2\|_F$.  Additionally, it is worthwhile to note $\sigma_{\min}(\hat{G}_{\{i\}})/\sqrt{r} \approx 1 - \sqrt{i/r}$ in expectation. The lower bound \eqref{eq:effecttail} thus suggests that a heavier tail $\Sigma_2$, in terms of $\|\Sigma_2\|_F$,  would result in larger values for $\sigma_i(AX)$ than a light tail. Consider in particular the last few singular values of $AX$, where the term $\sigma_{\min}(\hat{G}_{\{i\}})$ is small.  If the tail is light too, i.e.\ $\|\Sigma_2G_2\hat{x}\|$ is small, the lower bound will not be strong. We see in practice that this effect is indeed noticeable for the smallest singular values of $AX$, making them less reliable estimators. The bulk of the singular values of $AX$ are, however, not affected by the size of the tail given the low-rank structure of the matrix. We display this effect in Figures \ref{fig:effecttails1} and \ref{fig:effecttails2} with small-scale experiments.

\begin{figure}
\centering
\includegraphics[width = 1.0\linewidth]{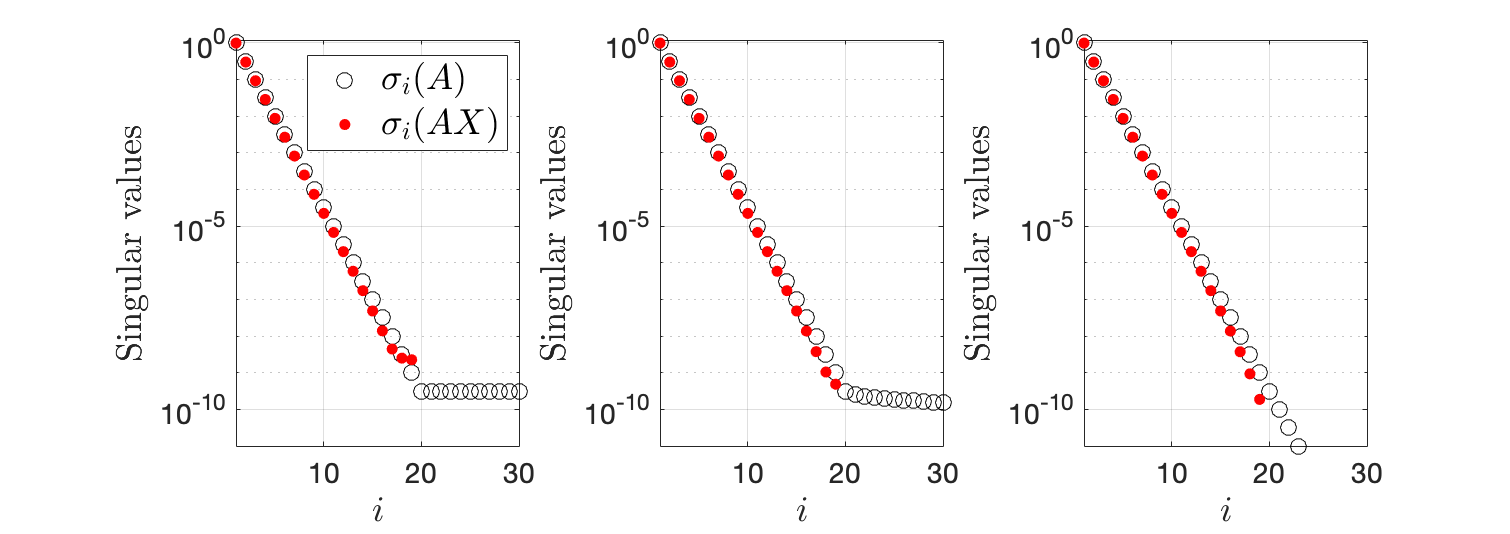}
\caption{How the tail of the singular value spectrum influences the singular value estimates. Each of the graphs shows the first 30 singular values $\sigma_i(A)$ of a square matrix of dimension $n=1000$. The first 20 singular values are identical for each of the matrices. The matrices in the graphs have constant, slow polynomially decreasing exponentially decreasing tails (from left to right). We sketch with a Gaussian embedding matrix of size $r= 19$. Although $\sigma_{r+1}=\sigma_{20}$ is the same for each of the matrices, we see the tail affects (only) the last few estimates. The results shown here are the average of 1000 trials, yet the behaviour is very typical.}
\label{fig:effecttails1}
\end{figure}
\begin{figure}
\centering
\includegraphics[width = 1.0\linewidth]{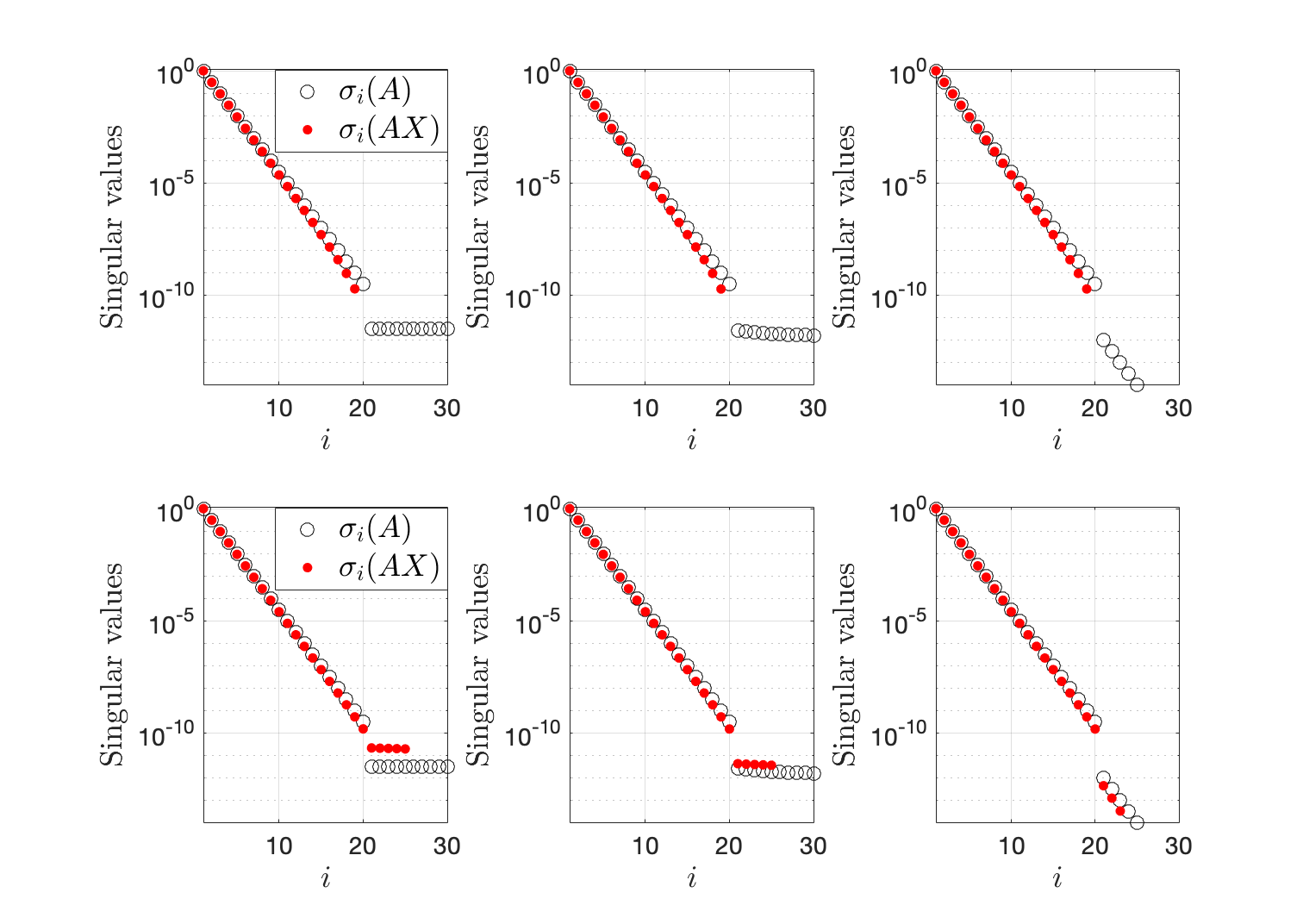}
\caption{One could argue that in none of the cases discussed in Figure \ref{fig:effecttails1} is the numerical rank very well-defined, because there is no clear gap between any of the singular values. We repeat the experiment of Figure \ref{fig:effecttails1} but with slightly different singular values. The top row shows the estimates of the first 19 singular values when there is a gap between the 20th and 21st singular value of $A$. We see that the effects of different tails becomes negligible. In the bottom row we repeat the experiment with $r=25$ instead of $r=19$. This illustrates that we can detect the gap in each of the cases. }
\label{fig:effecttails2}
\end{figure}

The small-scale experiments agree with bigger numerical results, as well as with theoretical results, in the sense that the extreme (smallest) singular value estimates are not trustworthy. For this reason we recommend oversampling by 10\% and disregarding the singular values corresponding to the oversampling. 


\subsection{Choice of sketch matrix}\label{sec:choicesketch}
In our analysis we have mainly focused on Gaussian embedding matrices, as they are the most well-studied class of random matrices and allow us to derive sharp constants.  A Gaussian embedding matrix is a random matrix with iid entries $\mathcal{N}(0,1/r)$. They are often the simplest sketch to implement. While it requires $O(mnr)$ operations to compute $AX$, 
when $r=O(1)$ the complexity is optimal as $mn$ is a lower bound for dense $A$, as clearly all entries of $A$ need to be read.  In fact, Gaussian sketches can be among the fastest to execute when $r=O(1)$.  Furthermore, Gaussian is the most efficient type of embedding in the sense that an embedding for an $r$-dimensional subspace can be achieved with high probability using $\mathcal{O}(r)$ samples.

Otherwise, if $r\gg 1$, other, more structured, classes of sketch matrices have been suggested and employed to speed up the computation of $AX$. These include subsampled randomized trigonometric transforms (SRTTs)~\cite{MartinssonTroppacta}, hashed randomized trigonometric transforms (HRTTs)~\cite{zhenshao}, and sparse matrix embeddings \cite{cohen2016nearly,clarkson2017low,kane2014sparser}.

In particular, SRTTs such as the subsampled randomized Hadamard, Fourier, discrete cosine or discrete Hartley transforms are random embeddings that allow fast application to a matrix.  The use of these sketches can be justified by Theorem~\ref{thm:embed}. While they come with weaker theoretical guarantees than Gaussian embeddings~\cite{Tropp2011ImprovedTransform} (in that the size of the sketch needs to be at least $\mathcal{O}(r\log r)$ to ensure an embedding with high probability),  empirical evidence strongly suggests that they usually perform just as well.  The performance of SRTT matrices is discussed more in the next section.

Related to SRTTs are the recently introduced class of random embeddings, HRTTs~\cite{zhenshao}. In these works it is shown an embedding for an $r$-dimensional subspace can be achieved with high probability using $O(r)$ samples (as with Gaussians) rather than $O(r\log r)$, by replacing the subsampling matrix $S$ in~\eqref{eq:Theta} with a hashing matrix (e.g.\ Countsketch~\cite{woodruff2014sketching}), which can be done without increasing the complexity. Computational results suggest HRTTs perform adequately in pathological cases where SRTTs fail, such as diagonal matrices  (see discussions on coherence in~\cite{Ipsen2014, candes2009exact}).

Finally, highly sparse embedding matrices have attracted much attention in the theoretical computer science literature~\cite{clarkson2017low,kane2014sparser}. The state-of-the-art result~\cite{cohen2016nearly} suggests an oversampling by a factor $\log r$ is needed to obtain an embedding with high probability.

When $A\in\mathbb{F}^{m\times n}$ has no structure to take advantage of (e.g. dense), such SRTT embeddings allow us to compute $AX$ with $O(mn\log n)$ operations, which is optimal up to an $O(\log n)$ factor. 
When $A$ has structure such as sparsity, $AX$ could be computed much more efficiently, for example in $O(\mbox{nnz}(A)r)$ operations for a Gaussian $X$ and less for sparse sketches. 

The choice of $X$ therefore depends on the structure of $A$; but the complexity of the sketching $AX$ can be bounded from above by $O(mn\log n)$ for dense matrices. We mainly use HRTT embedding matrices --- specifically hashed randomized DCTs --- in this paper as they are able to perform optimally, even for very coherent matrices. In our numerical experiments, we use diagonal matrices which is the most coherent type of matrix, and therefore most difficult to sketch.

\subsection{A numerical experiment}
We show how the theoretical result of this section, namely $\sigma_i(AX)/\sigma_i(A) = \mathcal{O}(1)$ for leading $i$, practically translates to a rank estimation technique. After sketching with an $n\times r$ embedding matrix $X$, we compute the singular values $\sigma_1(AX),\dots,\sigma_r(AX)$ and estimate $\text{rank}_{\epsilon}(A)$ for a given tolerance $\epsilon$ to be the first $\hat{r}$ such that $\sigma_{\hat{r}}(AX) > \epsilon\|A\| \geq \sigma_{\hat{r} + 1}(AX)$.  If (an estimate of) $\|A\|$ is unavailable, we may use $\sigma_1(AX)$. We applied this method to two synthetic example matrices, one diagonal matrix with polynomially decaying singular values ($\sigma_i(A) = i^{-3}$), and one diagonal matrix with exponentially decaying singular values ($\sigma_i(A) = 10^{-0.01(i-1)}$). The matrices are square of dimension $n = 10^5$.  

Figure \ref{fig:rankestAX} shows the results of the experiment. We plot the estimates for the numerical rank $\hat{r}$ that result from varying embedding dimensions $r$, and the ratio $\sigma_{\hat{r} + 1}(A)/\sigma_{\text{rank}_{\epsilon}(A) + 1}$. As both matrices do not have large gaps in the spectrum, the latter is an important quantity to judge the effectiveness of the rank estimator (see Section \ref{sec:whatisgoal}). We see that even for small values of $r$ and when using an HRTT or Gaussian embedding, $\sigma_{\hat{r} + 1}(A)$ is very close to $\sigma_{\text{rank}_{\epsilon}(A) + 1}$ and $\epsilon$ and well within the acceptable range described in Section 1.  Additionally, we see that SRTT matrices perform less than optimal for this pathological example of a diagonal matrix.

\begin{figure}
\hspace{-8mm}\includegraphics[width = 1.1\linewidth]{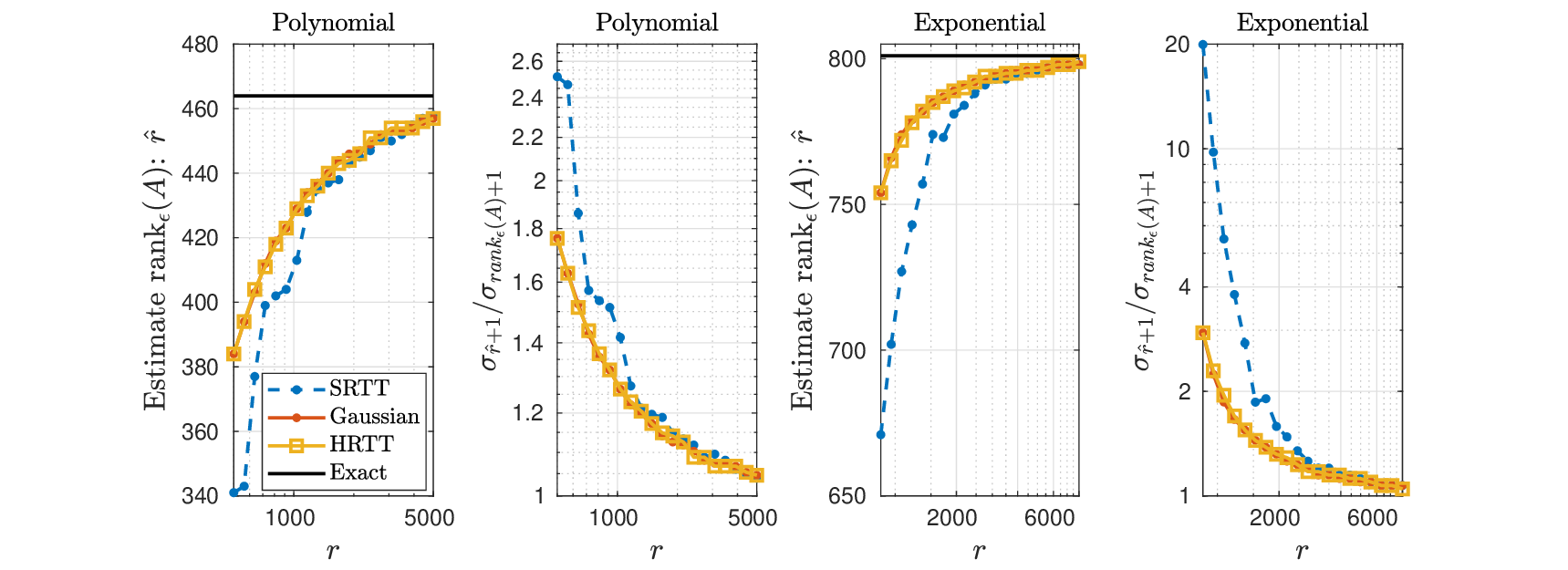}
\caption{Numerical rank estimation of matrices with polynomially decaying singular values and exponentially decaying singular values, where only the sketch $AX\in\R^{m\times r}$ is used. The matrices are real and square diagonal of dimension $n = 10^5$. The results are shown when using a subsampled randomized DCT embedding (SRTT), a Gaussian embedding matrix (Gaussian) and a hashed randomized DCT embedding (HRTT). The black lines indicate the exact numerical rank, $\text{rank}_{\epsilon}(A)$.}
\label{fig:rankestAX}
\end{figure}

\section{Randomized approximate orthogonalization: \boldmath$\sigma_i(Y AX)/\sigma_i(AX)=O(1)$} \label{sec:approxorth}
We have seen that the numerical rank of $AX$ is a good estimator for the numerical rank of $A$, provided that $A$ is approximately low-rank. Although $AX$ is of much smaller dimension than $A$, it may still be prohibitively expensive to compute its exact singular values. Furthermore, one could question whether it is worth calculating the singular values accurately spending $\mathcal{O}(mr^2)$ work, considering the accuracy that was lost in the sketching step. Here we describe a method, which we call randomized approximate orthogonalization, to cheaply compute estimates of the singular values of $AX$ in $\mathcal{O}(mr\log(m) + r^3)$ operations. The idea is similar to Section~\ref{sec:theory}, in that sketches roughly preserve (the leading) singular values, but as 
we are working with a tall-skinny matrix $AX$, 
here \emph{all} the singular values will be preserved up to a small multiplicative factor. 

Randomized approximate orthogonalization is inspired by the randomized least-squares solver framework named \emph{sketch-to-precondition}. The ideas were introduced in~\cite{rokhlin2008fast} and a fast implementation was introduced in Blendenpik~\cite{avron2010blendenpik}. In the framework,  a preconditioner for an overdetermined least-squares problem 
$\min_x\|Bx-b\|$ where $B\in\R^{m\times r_1}$ ($m\gg r_1$ is generated.  The solvers work as follows: the first step is to sample the rowspace of $B$ by sketching with a random embedding matrix $Y\in\R^{r_2\times m}$, where $r_2\geq r_1$, to obtain the small matrix $Y B$.  Secondly, one computes the QR factorization of $Y B = QR$. The upper triangular factor $R$ is then used as a preconditioner for an iterative solver such as LSQR. Even though the QR factorization is based on a sketch of $B$, it is now a well-known fact in RNLA that $BR^{-1}$ is `close to orthonormal’ in the sense that $\kappa(BR^{-1}) = \mathcal{O}(1)$.  Consequently, the singular values of $R$, or equivalently $Y B$,  are close to the singular values of $B$ (in the relative sense $\sigma_i(Y B)/\sigma_i(B)=O(1)$ for all $i$).  One could also view this process as a `randomized approximate orthogonalization' of $B$.

It follows that, in our context, we may estimate the singular values of $AX$ using the singular values of $YAX$, where $Y$ is another random embeddding (independent of $X$) of size $r_2\times m$.  It is then only necessary to compute the numerical rank (by computing the exact singular values) of the very small matrix $YAX\in\F^{r_2\times r_1}$ to estimate the numerical rank of the large matrix $A$. This is the final step of our rank estimation algorithm, which is discussed fully in the next section.

As for the choice of $Y$: unlike the first sketch of computing $AX$, in most cases the structure of $A$ (such as sparsity, if present) is lost in $B=AX$. It is therefore usually recommended that we take $Y$ to be an SRTT embedding, so that $YB = \Theta B$ can be computed in $O(mr\log m)$ operations. We assume this choice below and throughout the remainder of the paper, switching between the notations $YAX$ and $\Theta A X$ when appropriate.  That is,
\begin{equation}\label{eq:Theta}
	\Theta =\sqrt{ \frac{m}{r_2}}SFD,
\end{equation}
where $S\in\R^{r_2\times m}$ is a sampling matrix, $F$ a square orthogonal (or unitary) trigonometric transform (such as Fourier or DCT) of dimension $m$ and $D$ a diagonal matrix of independent random signs. 

Specifically the subsampled randomized Hadamard transform is analyzed extensively in~\cite{Tropp2011ImprovedTransform,    boutsidis2013improved}. These results can be readily extended to a general SRTT matrix and used in our context, as we show in the appendix in Lemma \ref{lem:BG}. This is essentially a restatement of~\cite[Thm.~3.1]{Tropp2011ImprovedTransform} and~\cite[Lem.~4.1]{boutsidis2013improved}. It leads to the following result on the accuracy of $\sigma_i(YAX)$ as estimates for $\sigma_i(AX)$.
\begin{corollary}\label{lem:YAX}
	Let $AX\in\R^{m\times r_1}$, with $m\geq r_1$, and let $\Theta\in\R^{r_2\times m}$ be an SRTT matrix as in~\eqref{eq:Theta}, where the trigonometric transform $F$ in $\Theta$ satisfies $\eta = m\max|F_{ij}|^2$. Let $0<\tilde{\epsilon}<1/3$ and $0<\delta<1.$ If
\begin{equation}
	6\eta\tilde{\epsilon}^{-2}\left[\sqrt{r_1} + \sqrt{8 \log(m/\delta)}\right]^2\log(r_1/\delta)\leq r_2\leq m,
\end{equation}
then, with failure probability at most $3\delta$
\begin{equation}
	\sqrt{1-\tilde{\epsilon}}\leq \frac{\sigma_i(\Theta AX)}{\sigma_i(AX)}\leq\sqrt{1+\tilde{\epsilon}},
\end{equation}
for each $i = 1,\dots, r_1$.
\end{corollary}

The use of the number $\eta = m \max|F_{ij}|^2$ is a natural choice as it only depends on the choice of trigonometric transform, not on the size of the matrices involved.  Choices for the trigonometric transform include Fourier, cosine, Hadamard, and Hartley transforms. The optimal $\eta$ is attained for a Hadamard transform, for which $\eta = 1$. The same holds for a Fourier matrix, but this involves operations in $\mathbb{C}$ and when $A$ is real, we may prefer real transforms. Since the Hadamard transform is only available when the dimension $m$ is a power of 2 (for which one solution is to append zeros to the matrix), we make use of the discrete cosine transform with $\eta = 2$ in our experiments when $A$ is real. Alternatively, we could have chosen the discrete Hartley transform with the same coherence $\eta = 2$.  See~\cite{avron2010blendenpik} for more discussion. 

We also note that a result analogous to Corollary~\ref{lem:YAX} is given in~\cite[Thm.~4.4]{meng2014lsrn} when $\Theta$ is a Gaussian matrix (and hence in a closer spirit to Section~\ref{sec:theory}). Results for a general subspace embedding matrix can be found in~\cite{woodruff2014sketching}.
\subsection{A numerical experiment}
We explained how \textit{all} singular values of a tall and skinny matrix $AX\in\F^{m\times r_1}$ can be estimated by the singular values of $YAX$, where $Y\in\F^{r_2\times m}$ is an embedding matrix and $r_2\geq r_1$. In the next section we conclude how this, combined with the results of the previous section, leads to a rank estimation algorithm. In the following numerical experiment, we first illustrate the error one can expect in practice in the step from $AX$ to $YAX$. 

We let $B$ be a dense matrix with fast polynomially decaying singular values, $\sigma_i(B) = i^{-3}$, of size $10^5\times 2000$.  We estimate $\sigma_i(B)$ with $\sigma_i(YB)$, where $Y$ is either a subsampled randomized DCT embedding or a Gaussian embedding of dimension $4000$. The relative error $|\sigma_i(YB) - \sigma_i(B)|/\sigma_i(B)$ is plotted in Figure \ref{fig:relerrorYA}.  We see the relative errors are $o(1)$ for each $i=1,\dots,2000$.  We use a dense matrix with incoherent left singular vectors to resemble the matrix $AX$, which will never be very coherent due to the prior application of embedding $X$. As a result, the SRTT embedding is sufficient.

\begin{figure}
\centering
\includegraphics[width = 0.9\linewidth]{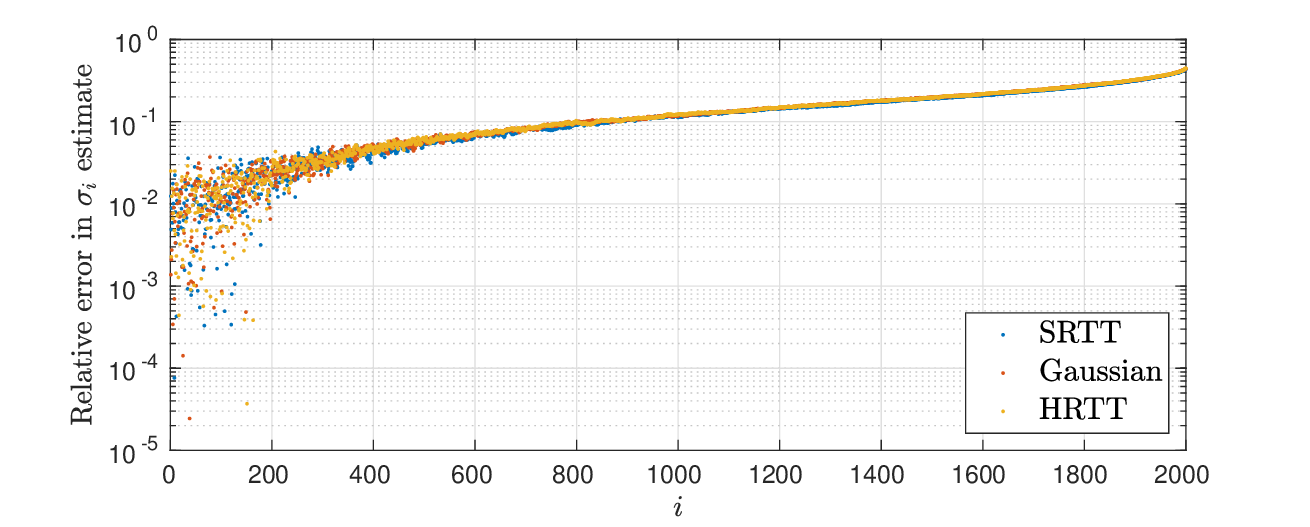}
\caption{Relative error in the singular value estimates of a real matrix $B$ with fast polynomially decaying singular values of size $10^5\times 2000$ based on the singular values of $Y B$, where $Y\in\R^{4000\times 10^5}$ is an embedding matrix (either a subsampled randomized DCT (SRTT), a Gaussian embedding matrix or a hashed randomided DCT (HRTT)). }
\label{fig:relerrorYA}
\end{figure}

\section{A numerical rank estimation scheme}\label{sec:mainalg}
The numerical rank estimation scheme that follows from the observations
\begin{enumerate}
\item $\sigma_i(AX) \approx \sigma_i(A)$ for an approximately low-rank matrix and leading singular values, and
\item $\sigma_i(YAX) \approx \sigma_i(AX)$ for all singular values of $AX$,
\end{enumerate}
is displayed in pseudocode in Algorithm \ref{alg:ouralg}. 

The algorithm requires an upper bound for the numerical rank as an input ($r_1$). This input informs the size of first embedding $X$, and is the total number of singular values estimates that will be obtained. If the input rank $r_1$ was not large enough, i.e.\ $r_1<\rank_\epsilon(A)$, Algorithm~\ref{alg:ouralg} is naturally only able to 
detect that $r_1$ is a lower bound for $\rank_\epsilon(A)$. 
To get a proper estimate, we then need to rerun the algorithm with a larger input rank, taking advantage of the preceding computation by appending extra columns to $X$ and rows to $\Theta$.

The choice of $r_1$ can be based on a number of possible considerations. 
For example, one might choose $r_1$ based on memory requirement, e.g. when one is unwilling to store more than $(m+n)r_1$ numbers. 
If some information is available on $A$ (e.g. it is updated from a matrix with known singular values), that can provide a good choice of $r_1$. If no information is available one might choose $r_1=O(1)$, say $r_1=64$; however, this can result in some inefficiency.\footnote{If the sketches employed are Gaussian, the complexity is $\mathcal{O}(mnr_1)$ (dominated by sketching) and the overhead of underestimating $r_1$ is minimal, bounded by a factor 2 (assuming the update rule $r_1:=2r_1$; other increments are of course possible). 
However, with other sketches such as SRTT, increasing $r_1$  
without resketching the whole matrix requires some care: For the SRTT $\mathcal{S}FD$, where $\mathcal{S}$ is a subsampling matrix, $F$ is a FFT/DCT matrix, and $D$ is a diagonal matrix of random signs, we store $FDAX_1$, $FDAX_2$, etc., and then subsample the big matrix $FDAX$.}

Provided the algorithm runs sucessfully, the computational complexity is $\mathcal{O}(mn\log n + r_1^3)$ for a dense matrix $A\in\mathbb{F}^{m\times n}$, where the two terms come from steps 3 and 7, assuming both $X$ and $Y$ are SRTT matrices. The cost of step 3 can be lower if $A$ has structure, in which case the $\mathcal{O}(mr_1\log m)$ cost in step 6 may become significant.

Regarding step 7, computing the singular values of $Y AX$, 
one can of course compute them via the SVD. We do this in most cases as the $O(r_1^3)$ cost is usually negligible. 
Alternatively,  since it suffices to retrieve estimates for $\sigma_i(YAX)$, one can compute the QR factorization $Y AX=QR$ and 
take the diagonal elements of $R$ as approximations of the singular values;  this works as a QR factorization of a matrix whose (right) singular vectors are randomized are rank-revealing with high probability~\cite{lawn237}.

\begin{algorithm2e}[H]\label{alg:ouralg}
\SetAlgoLined
\KwResult{Given an $m\times n$ matrix $A$, a tolerance $\epsilon$ (optional) and an estimate for the upper bound for rank $r_1$, this scheme computes an estimate for the $\epsilon$-rank of $A$.}
\textbf{1. Sketching:} \\
\nl Set $\tilde{r}_1 = \mbox{round}(1.1r_1)$. \\
\nl Draw $n\times \tilde{r}_1$ random embedding $X$.\\
\nl Form $AX$.\\
\textbf{2. Approximate orthogonalization:}\\
\nl Set $r_2 = 2\tilde{r}_1$. \\
\nl Draw $r_2\times m$ SRTT embedding $\Theta$. \\
\nl Form $\Theta AX$.\\
\textbf{3. Singular value estimates:}\\
\nl Compute the first $r_1$ singular values of $\Theta AX$.\\
\nl Let the numerical rank estimate be the first $\hat{r}$ s.t. $\sigma_{\hat{r} + 1}(\Theta AX) \leq\epsilon\|A\|$.  
If no such $\hat{r}$ exists, repeat the algorithm with a larger $r_1$, e.g. $r_1:=2r_1$,  by appending to the sketches.  If (an estimate of) $\|A\|$ is unavailable, use $\sigma_1(\Theta AX)$ to approximate $\|A\|$. \\
If no tolerance $\epsilon$ is provided, plot $\sigma_i(YAX)$ to detect a gap visually or consider measures as $\hat{r} = \max_i \sigma_i(YAX)/\sigma_{i+1}(YAX)$.
\caption{Rank Estimation with approximate orthogonalization.}
\end{algorithm2e}
\subsection{Rank estimation as part of randomized low-rank approximation}\label{sec:RERF}
It is worth discussing the case where Algorithm~\ref{alg:ouralg} is used as an initial step for randomized low-rank approximation, such as RSVD \cite{halko2011finding} or generalized Nystr\"om~\cite{nakatsukasa2020fast}. 
In such cases, computing the sketch $AX$ is needed anyway, so step 3 incurs no additional cost. 

Moreover, with generalized Nystr\"om, even $Y AX$ is required, as the low-rank approximant is of the form $AX(Y AX)^\dagger Y A$.
Furthermore, to evaluate $(Y AX)^\dagger$ a QR factorization $YAX=QR$ is computed. It follows that the additional work needed to execute Algorithm~\ref{alg:ouralg} is just to estimate the singular values of $R$. 
As described above, to do so one can simply take the diagonal elements of $R$, requiring just $O(r_1)$ work. 
Thus in the context of generalized Nystr\"om, a rank estimate can be obtained essentially for free. 

In the next section we explain how this rank estimate can be employed to speed up fixed-precision low-rank approximation. In particular, we note that using the singular value estimates we can reduce the size of a sketch as appropriate before performing the dominant computational work.

\subsubsection{Fixed-precision low-rank approximation problem}
Rank estimation as part of low-rank approximation is inherently linked to the fixed-precision problem, where one aims to compute a low-rank factorization $\hat{A}$ of unspecified rank such that $\|A - \hat{A}\| \leq \epsilon\|A\|$ for some user-specified tolerance $\epsilon$. We propose a scheme to combine rank estimation and low-rank approximation through existing error estimates. The algorithm will be used here for approximation in the Frobenius norm but can easily be extended to approximation in any other norm for which randomized low-rank approximation error estimates are available.  The aim will be to find an approximation of the form $\hat{A} = QB$, where $Q\in\F^{m\times r}$ has orthonormal columns and $r\leq n$. 

In their landmark paper of 2011, 
Halko, Martinsson and Tropp \cite{halko2011finding} describe and extensively analyze the randomized rangefinder (or randomized SVD) algorithm to compute such an approximation. This algorithm is displayed as Algorithm~\ref{alg:RF}.  Note that it can be extended to the randomized SVD algorithm by computing the SVD of $B = \hat{U}\Sigma V^T$ and then taking $U = Q\hat{U}$. A limitation of this factorization routine, as well as of other randomized low-rank approximation algorithms, is that we often wish to solve a fixed precision problem instead of a fixed rank algorithm, where the appropriate choice of the input rank $r$ is usually unknown. 

\begin{algorithm2e}[H]
\SetAlgoLined
\KwResult{Given an $m\times n$ matrix $A$, a target rank $r\geq 2$ and an oversampling parameter $p\geq 2$ such that $k+p\leq\min(m,n)$,  this scheme computes a $QB$-factorization of $A$.}
\nl Draw $n\times(r+p)$ random embedding matrix $X$.\\
\nl $[Q,\,\sim] = \text{qr}(AX,0)$. \quad (thin QR factorization)\\
\nl $B = Q^*A$.\\
\caption{Randomized Rangefinder \cite{halko2011finding}.}
\label{alg:RF}
\end{algorithm2e} 
As mentioned, we propose to use rank estimation to tackle this problem using a priori error estimates. Under the conditions of Algorithm \ref{alg:RF} with a Gaussian embedding matrix, it is known that \cite[Thm. 10.5]{halko2011finding}
\begin{equation}\label{eq:errRFfro1}
\mathbb{E}\|A - QB \|_F \leq \left( 1 + \frac{r}{p-1}\right)^{1/2}\left(\sum_{j>r}\sigma_j^2\right)^{1/2}.
\end{equation}
We can use the Rank Estimation Algorithm \ref{alg:ouralg} to compute estimates of the trailing singular values which appear in the estimate \eqref{eq:errRFfro1}. Suppose we have estimates $\hat{\sigma}_i = \sigma_i(YAX)\approx \sigma_i(A)$ for the first $r_1$ singular values of $A$, as a result of step 6 of Algorithm \ref{alg:ouralg}. Conservatively, and taking into account that there could be underestimation in the last few estimates if the tail is light, we set $\hat{\sigma}_i = \hat{\sigma}_{r_1}$ for $i = r_1 + 1,\dots,n$.  We can then choose the target rank $r$ in the rangefinder algorithm to be the first $r$ such that
\begin{equation*}
\left( 1 + \frac{r}{p-1}\right)^{1/2}\left(\sum_{j>r}\hat{\sigma}_j^2\right)^{1/2} \leq \epsilon\|A\|,
\end{equation*}
where $p$ is fixed to be 5 or 10 as suggested in \cite{halko2011finding} or where $p$ is fixed relative to $r$, say $p=0.1r$, as suggested in \cite{nakatsukasa2020fast}.

As mentioned before, the rank estimation part of this scheme requires only $\mathcal{O}(mr_1\log m + r_1^3)$ additional cost. More importantly, the singular value estimates it provides allow us to reduce the size of the sketch $AX$ ahead of the rangefinder process. This could result in considerable speed-up, as the QR factorization of $AX$ and the matrix-matrix multiplication $Q^*A$ are usually the dominant costs in the algorithm. Finally, we also know the rank of $QB$ is nearly minimal while (approximately) satisfying the accuracy criteria; by contrast, reducing the rank of a $QB$ factorization without rank estimation requires computing the SVD of $B$ and truncating it.

\begin{algorithm2e}[htbp]
\SetAlgoLined
\KwResult{Given an $m\times n$ matrix $A$, a tolerance $\epsilon$, an oversampling parameter $p \geq 2$ (either fixed absolutely or fixed relatively to $r$) and an upper bound for the numerical rank $r_1$, this scheme computes a QB factorization of $A$ such that $\|A- QB\|_F\lesssim\epsilon\|A\|$.}
\textbf{1. Rank Estimation:} \\
\nl Set $\tilde{r}_1 = \text{round}(1.1r_1)$, draw $n\times \tilde{r}_1$ random embedding $X$,  and form $AX$.\\
\nl Set $r_2 = 2\tilde{r}_1$, draw $r_2\times m$ SRTT $\Theta$, and form $\Theta AX$ \\
\nl Compute the first $r_1$ singular values of $\Theta AX$.\\
\nl Set $\hat{\sigma}_i = \sigma_i(\Theta AX)$ for $i=1,\dots, r_1$ and $\hat{\sigma}_i = \sigma_{r_1}(\Theta AX)$ for $i = r_1+1, \dots, n$.\\
\textbf{2. Randomized Rangefinder:} \\
\nl Set $r$ to be the smallest integer s.t. \vspace{-3mm}
\begin{equation}\label{eq:errRFfro}
\left( 1 + \frac{r}{p-1}\right)^{1/2}\bigg(\sum_{j>r}\hat{\sigma}_j^2\bigg)^{1/2} \leq \epsilon\|A\|.\vspace{-3mm}
\end{equation}
If no such $r$ exists, repeat the algorithm with 
a larger $r_1$, e.g. $r_1:=2r_1$, by appending to  the sketches.  If (an estimate of) $\|A\|$ is unavailable, use $\sigma_1(\Theta AX)$ to approximate $\|A\|$. \\
\nl $AX \leftarrow AX(:,\, 1:(r+p))$. \\
\nl $[Q,\,\sim] = \text{qr}(AX,0)$. \quad (thin QR factorization)\\
\nl $B = Q^*A$.
\caption{Rank Estimation combined with Randomized Rangefinder.}
\label{alg:RERF}
\end{algorithm2e}

\section{Numerical experiments}\label{sec:exp}
We introduce synthetic example matrices to experimentally support the theoretical results. Each of the matrices is diagonal and square of dimension $n = 10^5$.  They are, however, treated as dense matrices by the algorithms here. The singular values are on the diagonal in decreasing order.  We distinguish the following singular value spectra and resulting matrices, inspired by \cite{tropp2019streaming}:
\begin{enumerate}
\item Gaps in spectrum: we define a matrix $A_{G}$ with $\sigma_i =  1$ for $i = 1,\dots, 100$, $\sigma_i = 10^{-4}$ for $i = 101,\dots, 200$, $\sigma_i = 10^{-8}$ for $i = 201,\dots, 300$,  $\sigma_i = 10^{-12}$ for $i = 301,\dots, 400$ and $\sigma_i = 10^{-16}$ for $i > 400$.
\item Polynomial decay: each singular value takes the value $\sigma_i = i^{-p}$ for a parameter $p$.  We define a matrix $A_{SP}$ with a slow polynomial decaying spectrum (SP) with $p = 1$, and a matrix $A_{FP}$ with fast polynomial decay (FP) with $p = 3$. 
\item Exponential decay: the singular values are logarithmically equally spaced between 1 and $\sigma_n$, that is $\sigma_i = 10^{-q(i-1)}$ for a parameter $q$. We study a matrix $A_{SE}$ with slow exponential decay (SE) for which $q = 0.01$ and a matrix $A_{FE}$ with fast exponential decay (FE) for which $q = 0.5$. 
\end{enumerate} 
Each of the matrices has spectral norm $\|A\| = 1$ so the \textit{relative} accuracy notation will be suppressed in this section.
In each of the experiments in subsections 6.2-6.4, we will choose the embedding matrix $X$ to be a hashed randomized DCT matrix and $Y = \Theta$ to be a subsampled randomized DCT matrix. 
We display the perfomance of Algorithms \ref{alg:ouralg} and \ref{alg:RERF} in various numerical experiments.

\subsection{Detecting gaps in the spectrum}\label{sec:gaps}
Here we demonstrate that Algorithm \ref{alg:ouralg} works extremely well for approximately low-rank matrices that have gaps in their singular value spectrum. The algorithm can in fact be used in two ways: 1) if the specified tolerance is within the gap between two singular values, the algorithm will perform very well in identifying the correct rank even for a small number of samples, and 2) by considering the gaps in the singular value spectrum of $YA X$, one can identify gaps in the singular value spectrum of $A$ and use this to inform the $\epsilon$ value and target rank.  

To illustrate this we compare the singular values $\sigma_i(Y AX)$ to $\sigma_i(A)$ for various values of $r_1$ (see Algorithm \ref{alg:ouralg}). Note that we oversample by 10\% and disregard the singular values associated with oversampling. The results are shown in Figure \ref{fig:gaps}.  The figure displays the remarkable effectiveness of the algorithm, where gaps in the spectrum are clearly identified even for small values of $r_1$. 

We use Gaussian embeddings for the column space and SRCTs for the row space, to allow for an adaptive procedure. That is, we start with a Gaussian $X$ of dimension 110, compute $YAX$ and its singular values, and then increase the sketch dimension of $X$ by a 100 until we can identify the gap. This is how Algorithm~\ref{alg:ouralg} could be used if there is no upper bound on $r_1$ available.  We compare against the performance of the DOS algorithms~\cite{ubaru2016fast} (Chebyshev-based) and~\cite{Ubaru2017} (Lanczos-based).  Our rank estimation algorithm took 18.33 seconds to reveal the gap at the 400th singular value. The runtime of the DOS algorithms was 
2223 sec.  for~\cite{ubaru2016fast} and 199.1 sec. for \cite{Ubaru2017}. The DOS plots are also shown in Figure~\ref{fig:gaps}.

It is inherently difficult to detect gaps visible on a log-scale using a DOS algorithm.  The algorithms are based on estimating a sum of Dirac delta functions using a sum of smooth functions such as Gaussians. To be able to distinguish between very small eigenvalues, the width of the Gaussian has to be as small as the smallest eigenvalue. That is, it would be necessary to have a very high-resolution spectral density. The resolution refers to the size of the interval in which we estimate the number of eigenvalues~\cite{specdensity2016}.
However,  a fine resolution implies the number of sample points needed to detect larger eigenvalues is large. As a result, the computational effort necessary is not tractable. Figure~\ref{fig:gaps} shows that even if the DOS is estimated perfectly for the chosen Gaussian standard deviation, the gaps are not visible.

\begin{figure}
\centering
\includegraphics[width =\linewidth]{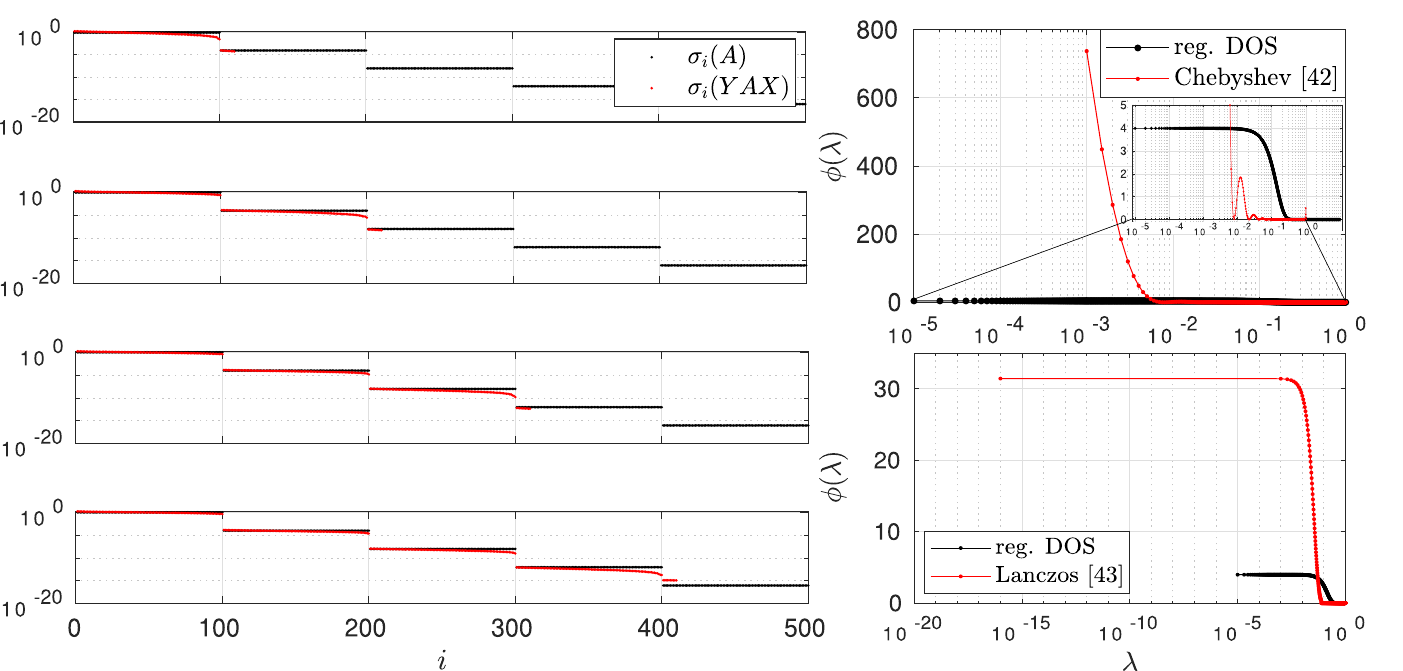}
\caption{On the left, the singular values of $A_G$ and $YA_{G}X$, where $X$ is a Gaussian embedding and $Y$ is an SRTT matrix with a discrete cosine transform (SRCT). The plots from top to bottom correspond to $r_1 = 110, 210,310,410$ respectively.  
The gaps that exist in the spectrum of $A$ are also visible in the spectrum of $YA_{G}X$, even for values of $r_1$ close to the location of the gaps.  On the right, the density of state plots resulting from the Chebyshev-based~\cite{ubaru2016fast} and Lanczos-based algorithms~\cite{Ubaru2017}. The algorithms are unable to identify the gaps for very small singular values. Even the exact regularized DOS can only identify the bulk of eigenvalues is of order $10^{-2}$ or smaller. See the references for further details on interpretation of these graphs.}
\label{fig:gaps}
\end{figure}

\subsection{Robustness of the rank estimator}
The other four example matrices discussed (SP, FP, SE, and FE) do not have significant gaps in their singular value spectra, which makes the numerical rank less well-defined.  Rather than focusing on the numerical rank estimate, $\hat{r}$, it thus makes sense to look at $\sigma_{\hat{r} + 1}(A)$ and its proximity to $\epsilon$.  Figure \ref{fig:robustness} shows the results of Algorithm \ref{alg:ouralg} applied 100 times to each of the example matrices.  We see that, although the rank estimate $\hat{r}$ we recover is often an underestimate when there are no significant gaps between singular values near the desired accuracy $\epsilon$,  $\sigma_{\hat{r} + 1}(A)$ is always very close to $\epsilon$. One could argue this is not necessarily a shortcoming of the algorithm, but rather a display of the difficulties associated to the concept of the exact numerical rank $\text{rank}_{\epsilon}(A)$ when $\sigma_r(A) \approx \sigma_{r+1}(A)$. Importantly, 
the algorithm always found a rank that satisfies the two desiderata set out in Section~\ref{sec:whatisgoal}. 
In particular, when there are significant gaps in the singular value spectrum we recover the exact numerical rank in all except one instance; see the rightmost column of the figure (and the forthcoming Figure~\ref{fig:gaps}). 
\begin{figure}
\centering
\includegraphics[width = \linewidth]{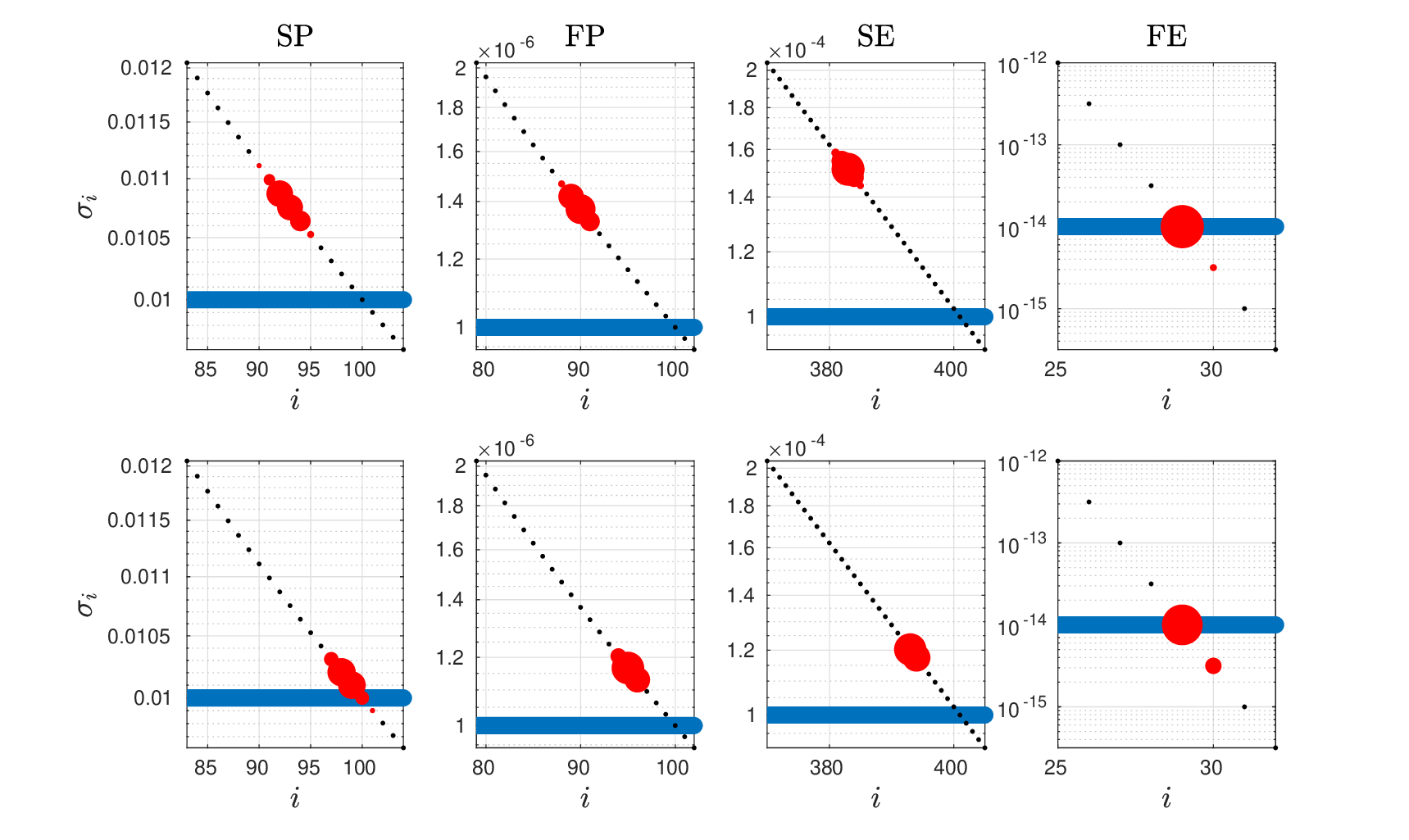}
\caption{The results of Algorithm 1 applied 100 times 
to four dense synthetic test matrices with singular value spectra that decay Slow Polynomially (SP), Fast Polynomially (FP), Slow Exponentially (SE), and Fast Exponentially (FE), as defined at the start of  Section \ref{sec:exp}. Each matrix is square diagonal with $n=10^5$. The blue line indicates the tolerance $\epsilon$ and the blacks dots represent the singular values $\sigma_i(A)$. The red dots indicate the singular values $\sigma_{\hat{r} + 1}(A)$, where $\hat{r}$ is the numerical rank estimate found by the algorithm. The area of the dot represents the frequency with which it a specific $\hat{r}$ was found.  The algorithm would work perfectly if the singular value on the blue line were found in each iteration, as is (almost) the case for the FE matrix.  The upper and lower row correspond to upper limits $r_1 = 2\,\text{rank}_{\epsilon}(A)$ and $r_1 = 4\,\text{rank}_{\epsilon}(A)$ respectively.}
\label{fig:robustness}
\end{figure}

\subsection{Comparison with other numerical rank estimation algorithms}\label{sec:expcompRE}
Suppose that the goal is to obtain the numerical rank of a given matrix for some tolerance $\epsilon$. We compare five different methods to achieve this for a $10^5\times 10^5$ matrix with slow exponential decaying spectrum (SE) and various values of $\epsilon$ ranging from $10^{-1}$ to $10^{-12}$.  The results are shown in Figure \ref{fig:numericalrankcomp}. 

The algorithms we compare are rank estimation (RE) with different parameters $r_1$,  the rank estimate obtained from RSVD with different parameters for $r$, the rank estimate obtained from randQB\_EI \cite{Yu2018efficient} with parameter $b=100$, the DOS algorithm with Chebyshev filters introduced in \cite{ubaru2016fast} (Cheb) and the Lanczos DOS algorithm introduced in \cite{Ubaru2017} (Lanczos), both with degree $50$ and $30$ samples, which are the default parameters in the authors' implementation\footnote{The implementation used is available at https://shashankaubaru.github.io/codes.html.}. 
We see that the latter two DOS algorithms do not accurately compute the rank for smaller values of $\epsilon$. A higher degree polynomial approximation would be necessary for this, which will slow down the computation. The DOS algorithms work very well for smaller matrices, but in the case of small tolerances and extremely large data, where each matrix-vector operation is costly in communication, the algorithms are outperformed by RE, RSVD, and randQB\_EI.\footnote{We used the implementation of randQB\_EI by the authors of \cite{Yu2018efficient}, available at  https://github.com/WenjianYu/randQB\_auto.}

As for these three, we see that RE is considerably faster in the context of rank estimation. Note also that the difference in running time will be more pronounced when the matrix is too large to fit into fast memory, as the other algorithms need to pass over the data more than once. Although the numerical rank is slightly underestimated for each value of $\epsilon$ with the Rank Estimation algorithm, and is exact for the other algorithms, we see that $\sigma_{\hat{r} + 1}$ is always extremely close to $\sigma_{\text{rank}_{\epsilon}(A) + 1}\approx \epsilon$.  This requirement for a good rank estimation set out in the introduction is thus satisfied. 

\begin{figure}
\hspace{-8mm}\includegraphics[width =1.1\linewidth]{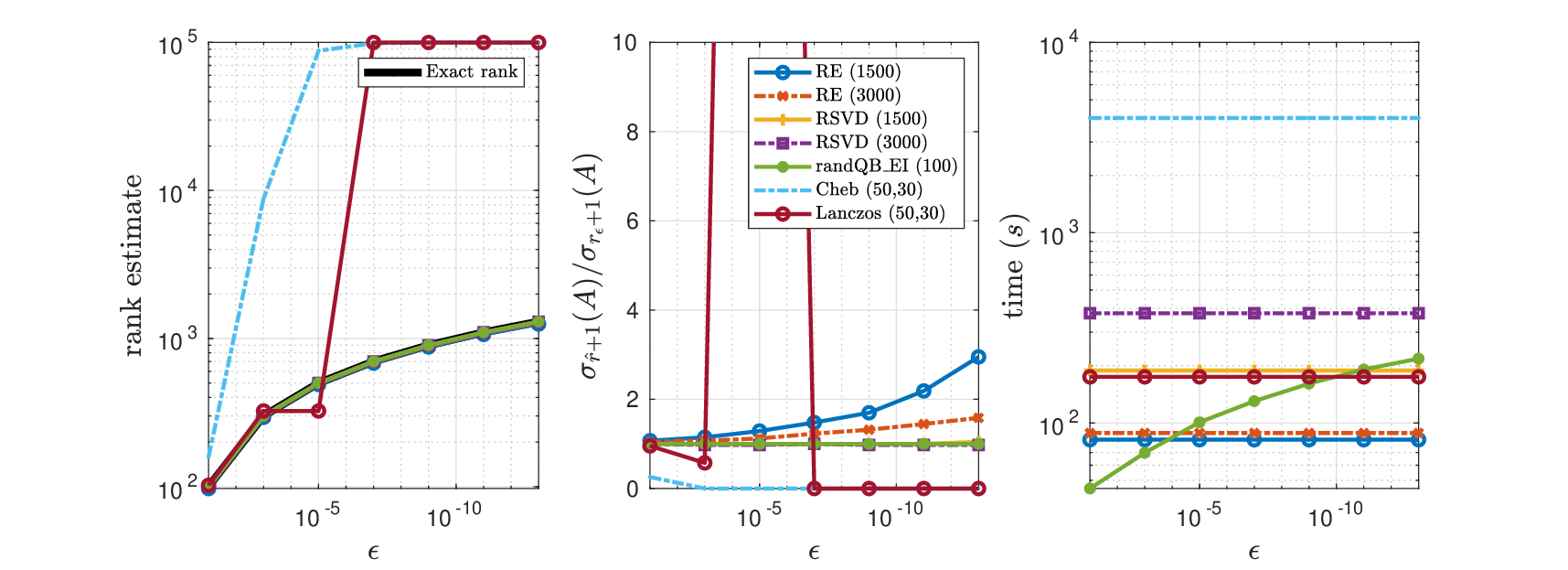}
\caption{We compare four algorithms to compute the numerical rank of a real square $10^5\times 10^5$ matrix with slow exponentially decaying singular values (SE).  The algorithms we compare are rank estimation (RE) with different parameters $r_1 = 1500,3000$,  the rank estimate obtained from RSVD with different parameters for $r=1500,3000$, the rank estimate obtained from randQB\_EI with parameter $b=100$, the DOS algorithm with Chebyshev filters introduced in \cite{ubaru2016fast} (Cheb) and the Lanczos DOS algorithm introduced in \cite{Ubaru2017} (Lanczos), both with degree $50$ and $\text{nvecs} = 30$. }
\label{fig:numericalrankcomp}
\end{figure}

\subsection{The fixed-precision problem}\label{sec:fixedprecision}
Often when dealing with large-dimensional approximately low-rank matrices, our aim is to obtain a low-rank approximation that is accurate within a given tolerance, as discussed in Section \ref{sec:RERF}. We compare three different algorithms to obtain such an approximation, Rank Estimation combined with Randomized Rangefinder (Algorithm \ref{alg:RERF}), Randomized Rangefinder (Algorithm \ref{alg:RF}), and the adaptive algorithm RandQB\_EI \cite{Yu2018efficient}). The results are shown in Figure \ref{fig:fixedprecision}.  The figure shows that the adaptive algorithm RandQB\_EI is the fastest in most cases.

\begin{figure}
\hspace{-8mm}\includegraphics[width =1.1 \linewidth]{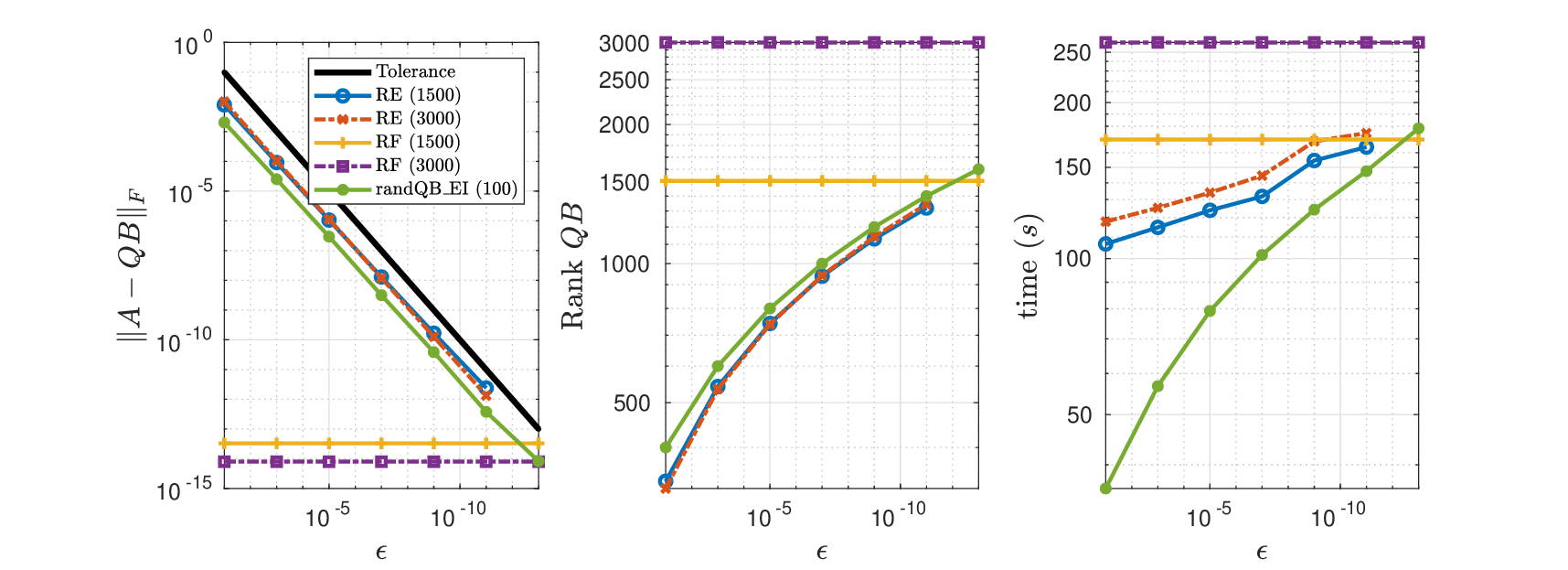}
\caption{ We compare four algorithms to compute a low-rank $QB$ approximation to solve the fixed precision problem $\|A_{SE} - QB\|\leq \epsilon$, where $A_{SE}$ is a $10^5\times 10^5$  square matrix with slow exponential decaying singular values.  The algorithms include rank estimation with rangefinder (RE), rangefinder (RF) and RandQB\_EI. We compare for the parameter choices $r_1 = 1500,3000$, $r=1500,3000$ and $b = 100$ for each of the algorithms respectively.  }
\label{fig:fixedprecision}
\end{figure}

It is difficult to compare these algorithms as the performance is very dependent on the context of the problem, and the parameters chosen.  We argue that our algorithm is the best choice when it is very expensive to communicate with the matrix, and a reliable prior rank estimate is unavailable. In this case, overestimating the target rank in randomized rangefinder and randomized SVD is much more expensive than it is in Algorithm \ref{alg:RERF}, rank estimation with rangefinder. This is portrayed in Figure \ref{fig:fixedprecisionoverest}.

\begin{figure}
\centering
\includegraphics[width =0.8\linewidth]{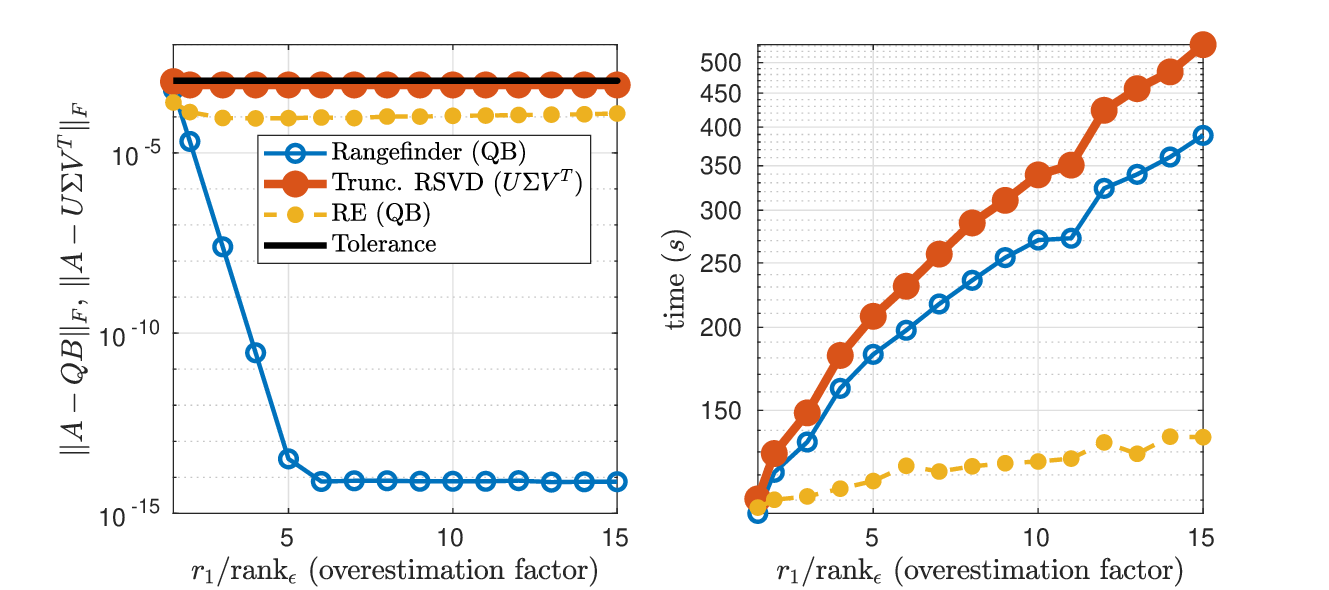}
\caption{We compare three algorithms for fixed precision estimation that require only two views of the matrix, randomized rangefinder (Algorithm \ref{alg:RF}), randomized SVD with truncation and rank estimation combined with randomized rangefinder (Algorithm \ref{alg:RERF}, shown as RE (QB)).  For a fixed tolerance $\epsilon = 10^{-3}$, we vary the input target rank/upper bound for the rank. Although each of the algorithms satisfy the condition of an error smaller than $\epsilon$, we see that Algorithm \ref{alg:RERF} is much faster, especially for greater overestimates. The matrix is $A_{SE}$ of dimension $n=10^5$. }
\label{fig:fixedprecisionoverest}
\end{figure}

\section{Discussion}
Let us conclude with a few extensions and open problems related to Algorithm~\ref{alg:ouralg}. 

As presented, Algorithm~\ref{alg:ouralg} does not take into account the structure of $A$ such as symmetry or sparsity. It would be of interest to devise a variant that respects such structure. 

When the input rank $r_1$ is insufficient, i.e. $r_1<\rank_\epsilon(A)$, Algorithm~\ref{alg:ouralg} needs to be rerun with a larger $r_1$. It would be helpful (if possible) to be able to find an appropriate new $r_1$ from the information obtained from the first run. 

While we primarily focused on rank estimation, the analysis indicates that many of the (leading) singular values of $A$ can be estimated by $\sigma_i(\Theta AX)$. 
Tuning each $\sigma_i(\Theta AX)$ to obtain a reliable (ideally unbiased) estimator is left for future work. 

\section*{Acknowledgement}
	We would like to thank the referees for their valuable comments, which improved the contents and exposition of the paper considerably. We also thank Nick Trefethen for valuable comments.

\appendix

\section{Additional Proofs}\label{sec:appendix}

\subsection{Proof of Corollary \ref{lem:YAX}}\label{sec:app2}
We wish to prove the following lemma from which Corollary \ref{lem:YAX} follows in a straightforward fashion:
\begin{lemma}\label{lem:BG}
	Let $U\in\R^{m\times r_1}$ have orthonormal columns and let $\Theta\in\R^{r_2\times m}$ be a SRTT matrix as in~\eqref{eq:Theta}, where the trigonometric transform 
$F$ in $\Theta$ satisfies $\eta = m\max|F_{ij}|^2$. Let $0<\\tilde{\epsilon}<1/3$ and $0<\delta<1.$ If
\begin{equation}
	6\eta\tilde{\epsilon}^{-2}\left[\sqrt{r_1} + \sqrt{8 \log(m/\delta)}\right]^2\log(r_1/\delta)\leq r_2\leq m,
\end{equation}
then, with failure probability at most $3\delta$
\begin{equation}
	\sqrt{1-\tilde{\epsilon}}\leq \sigma_{\min}(\Theta U)\leq\sigma_{\min}(\Theta U)\leq\sqrt{1+\tilde{\epsilon}}.
\end{equation}
\end{lemma}
The proof of Lemma \ref{lem:BG} consists of two main steps: first to show that premultiplying an orthonormal matrix with a randomized trigonometric transform smooths out information in the rows (or equivalently, equalizes the row norms), and second to show that sampling rows from a matrix with ‘smoothed out’ information works well. The proof is analogous to those by Boutsidis and Gittens~\cite{boutsidis2013improved}, but as their discussion only covers subsampled randomized Hadamard transforms instead of general SRTT matrices, we include the proof for completeness. 

A key concept in the proof is coherence, which is strongly related to the notion of `smoothing out' information in the rows. Let $A$ be a tall and skinny matrix and let $U$ be a matrix with orthonormal columns spanning the column space of $A$. Then the coherence of $A$ is defined
\begin{equation}
\mu(A) = \max \|U_{(i)}\|_2^2,
\end{equation}
where $U_{(i)}$ denotes the $i$th row of $U$. The minimal, and optimal, coherence of a matrix of size $m\times r$ is $\mu(A) = r/m$. This can be interpreted as all rows being equally important, which makes sampling rows more effective. The following lemma essentially states that applying a randomized trigonometric transform reduces the coherence of a matrix.
\begin{lemma}[Tropp \cite{Tropp2011ImprovedTransform}]\label{lemma:mixingcoherence}
Let $U\in\F^{m\times r}$ have orthonormal columns. Let $F$ be a square orthogonal matrix of dimension $m$ with $\eta = m\max|F_{ij}|^2$, $D$ be a $m\times m$ diagonal matrix of independent random signs, and define $\mathcal{F} = FD$. Let $0<\delta<1$. Then, with probability at least $1-\delta$
\begin{equation*}
    \mu(\mathcal{F}U) \leq \frac{\eta}{m}\left(\sqrt{r} + \sqrt{8 \log(m/\delta)} \right)^2.
\end{equation*}
\end{lemma}
We use the following lemma in the proof.
\begin{proposition}[\cite{Tropp2011ImprovedTransform}, Prop. 2.1]\label{prop.appproofcoherence} Suppose $f$ is a convex function that satisfies the Lipschitz condition
\begin{equation*}
    |f(x) - f(y)| \leq L\|x-y\| \quad \forall x,y.
\end{equation*}
Let $X$ be a vector of random signs. Then for all $t\geq 0$, 
\begin{equation*}
    \mathbb{P} \{f(X) \geq \mathbb{E}f(X) + Lt \} \leq e^{-t^2/8}.
\end{equation*}
\end{proposition}
The proof of the lemma then follows from defining such a convex, Lipschitz function.
\begin{proof}[Proof of Lemma \ref{lemma:mixingcoherence} \cite{Tropp2011ImprovedTransform}] Fix a row index $j\in\{1,2,\dots, m\}$ and define the function $$f(x) := \|e_j^*F\,\text{diag}(x)\,U\| = \|x^*EU\|,$$
where $E = \text{diag} (e_j^* F)$ is a diagonal matrix constructed from the $j$th row of $F$. Let $X$ be a vector of random signs. Since the diagonal of $D$ is also a vector of random signs, we can bound the norm of the $j$th row of $FDU$ by bounding $f(X)$ using Proposition \ref{prop.appproofcoherence}. We will first show that $f$ satisfies the conditions of Proposition \ref{prop.appproofcoherence}. Note that it is convex by the triangle inequality.  Since we have that $|F_{ij}|\leq\sqrt{\eta/m}$ for all $i,j$,  the entries of $E= \text{diag} (e_j^* F)$ all have magnitude less than or equal to $\sqrt{\eta/m}$ as well,  which means $\|E\|\leq \sqrt{\eta/m}$.  We use this to show the Lipschitz constant of $f$ is $L = \sqrt{\eta/m}$.
\begin{align*}
    |f(x) - f(y)| &=  | \|x^*EU \| - \|y^*EU \| | \leq \|(x-y)^*EU\| \leq  \|E\| \|x-y\|
\end{align*}
We bound the expected value of $f(X)$ as
\begin{align*}
    \mathbb{E}f(X) &\leq \left[\mathbb{E}f(X)^2 \right]^{1/2}  = \left[\mathbb{E}\|X^*EU \|^2 \right]^{1/2} = \|EU\|_F  \leq \|E\|\|U\|_F  \leq \sqrt{\frac{\eta}{m} r}.
\end{align*}
Now apply Proposition \ref{prop.appproofcoherence} with $t = \sqrt{8\log(m/\delta)}$ to find
\begin{equation*}
    \mathbb{P} \left\{\|e_j^*FDU\| \geq \sqrt{\frac{\eta}{m} r } + \sqrt{\frac{\eta}{m} 8\log(m/\delta)} \right\} \leq \frac{\delta}{m}.
\end{equation*}
We need this to hold for each of the $m$ rows.  By the union bound, we obtain
$$\mathbb{P} \left\{\max_{1\leq j\leq m}\|e_j^*FDU\| \geq \sqrt{\frac{\eta}{m}}(\sqrt{r} + \sqrt{8\log(m/\delta)}) \right\} \leq m\frac{\delta}{m} = \delta.$$
Finally, note that $\mu(FDU) = \max\nolimits_j\|e_j^*FDU\|^2$.
\end{proof}
We need to combine this result with Lemma 4.3 of \cite{boutsidis2013improved} which is a corollary to Lemma 3.4 of \cite{Tropp2011ImprovedTransform}. The proof can be found in the references. 
\begin{lemma}[Boutsidis and Gittens~\cite{boutsidis2013improved}]\label{lemma:samplingcoherence}
Let $U\in\F^{m\times r_1}$ have orthonormal columns and let $\mu(U)$ denote its coherence. Let $0<\tilde{\epsilon}<1$ and $0<\delta < 1$. Let $r_2$ be an integer such that
\begin{equation*}
    6\tilde{\epsilon}^{-2}m\mu(U)\log(r_1/\delta) \leq r_2 \leq m.
\end{equation*}
Let $S\in \R^{r_2\times m}$ sample $r_2$ rows uniformly at random and without replacement, and define $\mathcal{S} = \sqrt{\frac{m}{r_2}}S$. Then, with probability at least $1-2\delta$
\begin{equation*}
    \sqrt{1-\tilde{\epsilon}}\leq\sigma_{\min}(\mathcal{S}U)\leq\sigma_{\max} (\mathcal{S}U) \leq \sqrt{1+\tilde{\epsilon}}.
\end{equation*}
\end{lemma}

\bibliographystyle{abbrv} 
\bibliography{bib2}

\end{document}